\documentclass[twocolumn,amsthm,secthm]{autart}

\usepackage{mathtools} 
\usepackage{amssymb}
\usepackage{amsbsy}
\usepackage{color}
\usepackage{verbatim} 

\usepackage{graphicx}
\usepackage{epstopdf}
\usepackage{bmpsize}

\usepackage[scriptsize,tight,center]{subfigure}

\theoremstyle{plain}
\newtheorem{theorem}{Theorem}[section]
\newtheorem{lemma}[theorem]{Lemma}
\newtheorem{proposition}[theorem]{Proposition}
\newtheorem{remark}[theorem]{Remark}

\newtheorem{problem}{Problem}
\theoremstyle{definition}
\newtheorem{definition}[theorem]{Definition}
\newtheorem{assumption}[theorem]{Assumption}
\theoremstyle{remark}

\newcommand{\Romannum}[1]{{_{\uppercase\expandafter{\romannumeral#1}}}} 

\newcommand{\rev}[1]{\textcolor{black}{#1}}

\usepackage{pmat} 
\usepackage{enumitem} 
\usepackage{float}

\begin{document}

\begin{frontmatter}

\title{An Intrinsic Approach to Formation Control of Regular Polyhedra for Reduced Attitudes\thanksref{footnoteinfo}}

\thanks[footnoteinfo]{This paper was not presented at any IFAC
meeting.
}

\author[kth]{Silun Zhang}\ead{silunz@kth.se},
\author[hit]{Fenghua He}\ead{hefenghua@hit.edu.cn},
\author[amss]{Yiguang Hong}\ead{yghong@iss.ac.cn},
\author[kth]{Xiaoming Hu}\ead{hu@kth.se}

\address[kth]{Optimization and Systems Theory, Department of Mathematics, KTH Royal Institute of Technology, 100 44 Stockholm, Sweden}
\address[hit]{Control and Simulation Center, Harbin Institute of Technology, 150001 Harbin, China}
\address[amss]{Key Laboratory of Systems and Control, Academy of Mathematics and Systems Science, Chinese Academy of Sciences, 100190 Beijing, China}

\begin{keyword}
Attitude control; distributed control; formation control; nonlinear systems.
\end{keyword}

\begin{abstract}                          
This paper addresses formation control of reduced attitudes in which a
continuous control protocol is proposed for achieving and stabilizing all regular polyhedra (also known as Platonic solids) under a unified
framework. The protocol contains only relative reduced attitude measurements and does not depend on any particular parametrization as is usually used in the literature. A key feature of the control proposed is that it is intrinsic in the sense that it does not need to incorporate any information of the desired formation. Instead, the achieved formation pattern is totally attributed to the geometric properties of the space and the designed inter-agent connection topology. Using a novel coordinates transformation, asymptotic stability of the desired formations is proven by studying stability of a constrained nonlinear system. In addition, a methodology to investigate stability of such constrained systems is also presented.
\end{abstract}

\end{frontmatter}

\section{Introduction}
In the last decades coordination of multi-agent systems has emerged as
a significant research topic across the control communities.
This research tendency evolves from synchronization further towards more flexible collective behaviors, among which cooperative formation is an important one with a diverse range of engineering applications, such as formation flying \cite{porfiri2007tracking,Ren03ACC}, sensor placement \cite{cortes2002coverage}, and spatial exploration \cite{curtis1999magnetospheric}. A similar development trend has also taken place in the attitude control area.

Attitude control has many applications and was partly motivated by
aerospace developments in the middle of the last century~\cite{Bower1964,Kowalik1970}.
A well-known result on the controllability of attitude systems states
that no continuously differentiable feedback control can
asymptotically stabilize the attitude of a spacecraft with only two
actuators \cite{Byrnes1991}. However, in this under-actuated scenario
a smooth feedback can be derived to entail asymptotic stability with
respect to two axis and rotating about the third axis on the
closed-loop system \cite{Tsiotras1994}. Inspired by this two-axis
stabilization of attitude systems, \cite{bullo1995control} proposes
the \emph{reduced attitude} control problem, where only the pointing
direction of a body-fixed axis is considered, while any rotation about
this axis is ignored. This model is then shown to be a proper
framework for many applications, such as control of antenna orientation for satellites
\cite{Tsiotras1994,chaturvedi2011rigid} and viewing field for cameras \cite{Wang13PRL}.
Another reason for the name
of reduced attitude is that to the contrary of full attitude that
evolves in a 3-dimensional Lie group $\mathcal{SO}(3)$, reduced attitude has one less degree of freedom with configuration space $\mathcal{S}^2$.

In attitude control study, there has been an increasing research interest in
attitude formation missions. Based on a
spherical parametrization, \cite{paley2009stabilization} addresses the
formation problem in $\mathcal{S}^2$, in which absolute state
measurements are however required.
 \cite{lawton2002synchronized} and \cite{min2011distributed} propose
 a leader-follower formation control scheme based on the
 parametrizations of unit-quaternions and modified Rodrigues
 parameters respectively, but the relative errors between leaders and
 followers are identified by the difference of their parametrization
 variables, and the control implementation also needs the absolute
 attitude information.
 \rev{To overcome the drawback caused by parametrizations, \cite{WJ15ac} provides a reduced-attitude formation control scheme directly in $\mathcal{S}^2$ space. Moreover, such control protocol is in a so-called intrinsic manner which does not require any formation errors in control law and the desired formation patterns are constructed totally based on the geometric properties of the configuration space and the designed connection topology. }

Platonic solids are the only five regular polyhedra in 3-dimensional space, of which formations control have many promising applications \cite{roscoe2013satellite,gracias2000forming,hughes2008formation}.
\rev{
This is because Platonic solids possess the most symmetries among all polyhedra, which leads to that such formations can facilitate the achievement of maximal observational effectiveness \cite{roscoe2013satellite}, for example in NASA's Magnetospheric Multiscale mission \cite{hughes2008formation}, the Glassmeier's quality metric \cite{dunlop1998multi} is maximal when the four spacecrafts form a regular tetrahedron configuration. Another engineering application of regular tetrahedron formation is the Cluster II mission launched by the ESA \cite{credland1997cluster}. Moreover, in computer vision field regular polyhedron configurations are also used in motion capture applications with a panoramic field of view \cite{pless2003using,kim2010spherical,goldberger2005reconstructing}, for example when the orientations (reduced attitudes) of six cameras forms a regular octahedron known as the "Argus Eye" system \cite{pless2003using}, more accurate estimation for target's motion is obtained.
}

\rev{
In this paper, we continue our work \cite{zhang2017tetrahedron} to further achieve the intrinsic formation control
for all regular polyhedra in $\mathcal{S}^2$ under a unified framework.
Compared with the existing results, the main contributions of the current work focus on:
\vspace{-15pt}
\begin{enumerate}
\item A unified characterization of regular polyhedra is proposed based on the particular rotational symmetries (so-called polyhedral group) possessed by such solids;
\item Since in intrinsic formation control the desired configuration is partially encoded in the inter-agent topology, we propose a symmetry-based methodology for designing suitable inter-agent graphs, by which a family of possible topologies are accordingly given;
\item In contrast with \cite{zhang2017tetrahedron}, under the proposed coordinates the resulting dynamics entails far more algebraic constraints on the state space.
To address such a problem, the concept of exponential stability of a system restricted to a manifold is proposed, which is a systematic approach to obtaining exponential stability of a system subject to algebraic constraints.
\end{enumerate}
\vspace{-15pt}
By leveraging these tools, we show that it is sufficient to investigate stability of a  far less constrained system to obtain stability of each desired formation.
Besides, it is worthwhile to mention that the proposed
control protocol only contains relative reduced attitude measurements,
and does not depend on any parametrization usually used in the
attitude control.
}


\section{Notation and Preliminary}\label{sec:pre}
In this paper, the topology of inter-agent connectivity is modeled by
a graph $\mathcal{G}=(\mathcal{V},\mathcal{E})$, where the set of
nodes is $\mathcal{V}=\big\{1,\ldots,N\big\}$, and $\mathcal{E}\subset \mathcal{V} \times \mathcal{V}$ is the edge set. A graph $\mathcal{G}$ is said to be undirected if $(j,i) \in \mathcal{E}$, for every $(i,j) \in \mathcal{E}$. The adjacency matrix of an undirected graph $\mathcal{G}$ is defined by a matrix $A_\mathcal{G}=\big[a_{ij}\big]_{i,j \in \mathcal{V}} \in \mathbb{R}^{N \times N}$ such that its entry $a_{ij}=1$ if $(j,i) \in \mathcal{E}$, otherwise $a_{ij}=0$. We also define the neighbor set of node $i$ as $\mathcal{N}_i= \big\{ j:(j,i) \in \mathcal{E}\big\}$, and we say $j$ is a neighbor of $i$, if $j \in \mathcal{N}_i$. We denote $Card(S)$ as the cardinality of a set $S$ and $\mathbf{I}_{n}$ as the identity matrix with dimension $n$. The symbol $\mathbb{Z}_n$ is reserved for the integer set $\{1, \dots, n\}$ and $GL(n,\mathbb R)$ for the general linear group of degree $n$ over $\mathbb R$.

\subsection{Attitude and Reduced Attitude}
In this work, we consider the formation problem for reduced attitudes of $N$ rigid bodies.
\rev{
The \emph{reduced attitude} is devoted to the applications wherein
the pointing direction of a body-fixed axis is concerned. Let
$\boldsymbol{b}_i\in \mathcal{S}^2$ denote the coordinates of agent
$i$'s pointing axis relative to the body frame $\mathcal{F}_i$, where $\mathcal{S}^2=\big\{x\in\mathbb{R}^{3}: \|x\|=1\big\}$.
Then the coordinates of $\boldsymbol{b}_i$ relative to the inertial frame $\mathcal{F}_w$ is $\Gamma_i=R_i\boldsymbol{b}_i$, in which $R_i \in \mathcal{SO}(3)$ is the attitude matrix of agent $i$ relative to $\mathcal{F}_w$.
This $\Gamma_i \in \mathcal{S}^2$ specifies the pointing direction of axis $\boldsymbol{b}_i$ and has one dimension less than the full attitude, thus is said to be the \emph{reduced attitude} of rigid body $i$.
}

 In the inertial frame $\mathcal{F}_w$, the kinematics of the reduced attitude $\Gamma_i$ is governed by \cite{ZXLi}
 \begin{equation}\label{eq:pre:DGamma}
        \dot{\Gamma}_i = \widehat{\omega}_i \Gamma_i ,
\end{equation}
where $\omega_i \in \mathbb{R}^{3}$ is agent $i$'s angular velocity relative to frame $\mathcal{F}_w$, and the hat operator $\widehat{(\cdot)}$ is the linear operator of the cross product defined as $\widehat x y=x\times y$, for any $x, y\in \mathbb{R}^{3}$.
%
We note that $\widehat{x} \in \mathfrak{so}(3)$, where $\mathfrak{so}(3)$ is the Lie algebra of $\mathcal{SO}(3)$ consisting of all skew symmetric matrices.
%

For any two points $\Gamma_i, \Gamma_j \in \mathcal{S}^2$, we define angle $\theta_{ij} \in [0, \pi]$ and vector $k_{ij} \in \mathcal{S}^2$ as
$$ \theta_{ij}=\arccos(\Gamma_i^T \Gamma_j), \;\; k_{ij}=\frac{\widehat{\Gamma}_i \Gamma_j}{\sin(\theta_{ij})}.$$
In the definition of $k_{ij}$, we stipulate $k_{ij}$ to be any unit vector orthogonal to $\Gamma_i$ when $\theta_{ij}=0$ or $\pi$.

We note that $\theta_{ij}$ is as well the geodesic distance between $\Gamma_i$ and $\Gamma_j$ in $\mathcal{S}^2$, and we have $\Gamma_j= \exp(\theta_{ij} \widehat{k}_{ij})\Gamma_i$ for $\Gamma_i, \Gamma_j \in \mathcal{S}^2$.
The next formula states the relationship between three reduced attitudes, which is also referred to as the spherical cosine formula \cite{Todhunter_Identity}:
\begin{lemma}\label{Thm:PR:cosine}
  For any three reduced attitudes $\Gamma_i, \Gamma_j$, $\Gamma_k \in \mathcal{S}^2$, the following relationship always holds:
$$\cos(\theta_{ij})=\cos(\theta_{ik})\cos(\theta_{jk})+\sin(\theta_{ik})\sin(\theta_{jk})k_{ik}^Tk_{jk}.$$
\end{lemma}

In this paper, we will also use a frequently mentioned parametrization of $\Gamma_i$ based on the RPY angles system \cite{Sciavicco_Book},
\begin{gather}\label{eq:pre:RPY}
  \Gamma_i=[\cos(\psi_i)\cos(\phi_i),\sin(\psi_i)\cos(\phi_i),\sin(\phi_i)]^T
\end{gather}
where $\psi_i \in [-\pi,\pi)$, $\phi_i \in [-\pi/2,\pi/2]$.

\subsection{Regular Polyhedra}\label{Sec:fiveRegularPolyhedra}
A convex polyhedron is said to be regular if its faces are identical regular polygons and its vertices are all surrounded by the same pattern. The regular polyhedra can be identified by the \emph{Schl\"{a}fli symbol}, according to which a polyhedron with the $p$-sided regular polygon faces and the vertices surrounded by $q$ such faces is denoted by $\{p, q\}$, where $p,q \in \mathbb{Z}$. At every vertex of $\{p,q\}$, there are $q$ face-angles in the size of $\pi-2\pi/p$  \cite{coxeter1973regular}, and the sum of these $q$ angles are less than $2\pi$. Therefore, we have the inequality
\begin{equation}\label{eq:pre:relation_p_q}
1/p+1/q>1/2.
\end{equation}
This inequality leads to that there exist only five possible combinations of $p$ and $q$
as shown in Fig. \ref{Fig:Int:PlatonicSolid}. They are also referred to as the five \emph{Platonic solids}.


We denote the number of vertices, edges and faces of $\{p,q\}$ as $N_0^{\{p,q\}}$,$N_1^{\{p,q\}}$ and $N_2^{\{p,q\}}$ respectively. For simplicity, if there is no ambiguity in the context, we omit the superscript $\{p,q\}$ in these notations. By Euler's formula \cite{peter1999polyhedra}, we have $N_0=4p/d$, $N_1=2pq/d$ and $N_2=4q/d$, where $d=4-(p-2)(q-2)$.

\subsection{Permutation and Permutation Matrix}
\rev{
Given a finite set $S$, we define a permutation $\sigma$ of $S$ as a bijective mapping from $S$ to itself, i.e., $\sigma: S\to S$. Let $\sigma$ and $\pi$ be two permutations of $S$, then the product $\sigma \cdot \pi $ is defined by $\sigma \cdot \pi (s)=\sigma(\pi(s))$, $\forall s \in S$. Endowed with the operation of such a product, the class of all permutations of a finite set forms a group.
}

\rev{
We denote a \textit{cycle} of permutation $\sigma$ as $(s_1, s_2, \dots, s_m)$ which is a group orbit satisfying $s_i \in S$, $s_{i+1}=\sigma(s_i)$ for $i=1, \dots, m-1$ and  $s_{1}=\sigma(s_m)$. Two cycles are said to be disjoint if they do not have any common elements. It can be shown that any permutation on a finite set admits a unique cycle decomposition consisting of mutually disjoint cycles whose union is $S$. For this reason, permutation $\sigma$ can be identified as a product of disjoint cycles which is called \textit{cycle notation}. For example, the notation $(1)(2, 3, 4)$ presents the permutation defined by $\sigma(1)=1, \sigma(2)=3, \sigma(3)=4, \sigma(4)=2$.
}

\rev{
For a permutation specified by mapping $\sigma: S \to S$, we define its permutation matrix by $P_\sigma=[\mathbf{e}_{\sigma(1)},\cdots,\mathbf{e}_{\sigma(N_0)}]^T$, where $\mathbf{e}_i$ represents the $i$-th column of identity matrix $\mathbf{I}_{Card(\mathcal{S})}$. We note that the permutation matrix $P_\sigma$ is an orthogonal matrix, i.e. $P_\sigma P_\sigma^T=\mathbf{I}$.
}

\section{Reduced Attitude Control}\label{sec:pformu}
In this paper, we focus on an intrinsic formation control for reduced
attitudes, which implies that in contrast to most existing work the control protocol contains no formation error which is the difference of the current formation from the desired one. Instead, the constructed formation pattern is totally attributed to the geometric properties of the compact manifold $\mathcal{S}^2$ and the designed connection topology.

In our previous work \cite{WJ15ac}, an intrinsic control law only containing the relative attitude $\{\widehat{\Gamma}_i\Gamma_j:j\in \mathcal{N}_i\}$ is proposed to reach antipodal and cyclic formations under the ring-graph topology. Here, a similar but slightly modified control is employed for Platonic solid formations as
\begin{equation}\label{eq:PF:Control}
  \omega_i=-\sum_{j\in \mathcal{N}_i} h(\theta_{ij})\widehat{\Gamma}_i\Gamma_j, \;\quad i\in \mathcal{V},
\end{equation}
where $\mathcal{V}=\{1, \cdots, N_0\}$, $h: \mathbb{R} \to \mathbb{R}$ is a real functional satisfying that the function composition 
$\bar{h}=h \circ \arccos$ is Lipschitz.
Substituting (\ref{eq:PF:Control}) into the kinematics (\ref{eq:pre:DGamma}), the closed-loop system reads
\begin{equation}\label{eq:PF:ClosedSystem}
  \dot{\Gamma}_i=\widehat{\Gamma}_i \sum_{j\in \mathcal{N}_i} h(\theta_{ij})\widehat{\Gamma}_i\Gamma_j, \; i\in \mathcal{V}.
\end{equation}

\rev{
We note that control law (\ref{eq:PF:Control}) is independent of any information of global initial frame $\mathcal{F}_w$. In practice the controller of a rigid body is almost always implemented in the body frame, since rotational actuators, such as momentum wheels, are always installed fixed to the body. In the body frame, we have
\begin{equation}\label{eq:PF:controlinbody}
\omega_i^b=-\sum\limits_{j\in N_i} h(\theta_{ij}) b_i \times (R_i^TR_jb_j),
\end{equation}
where $\omega_i^b$ is the angular velocity of body $i$ relative to the inertial frame $\mathcal{F}_w$ resolved in the body frame $\mathcal{F}_{b_i}$. Note that when we consider the control in $\mathcal{F}_{b_i}$, the kinematics \eqref{eq:pre:DGamma} reads $\dot{\Gamma}_i=(R_i\widehat{\omega}_i^b) b_i$. If we plug in $\omega_i^b$ , we have
\begin{equation*}
  \dot{\Gamma}_i \!=\! R_i \widehat{b}_i \sum\limits_{j\in N_i} h(\theta_{ij}) b_i \times (R_i^TR_jb_j)\!=\! \widehat{\Gamma}_i \sum\limits_{j\in N_i}  h(\theta_{ij})\widehat{\Gamma}_i {\Gamma}_j.
\end{equation*}
which is exactly the closed-loop system in \eqref{eq:PF:ClosedSystem}. We can see in body frame $\mathcal{F}_{b_i}$ control \eqref{eq:PF:controlinbody} only contains relative information between the reduced attitudes and can be measured from a local frame.
}

By Rodrigues' rotation formula, the following lemma shows that the closed-loop system (\ref{eq:PF:ClosedSystem}) is invariant under any rotations.
\begin{lemma}\label{Thm:PF:rotationalInvariant}
For any rotation transformation about a unit axis $u \in \mathcal{S}^2$ through an angle $\theta \in [0,\pi]$, the system  (\ref{eq:PF:ClosedSystem}) is invariant, i.e. if
$\Pi_i = \exp(\theta \widehat{u})\Gamma_i, \; i \in \mathcal{V}$
then the closed-loop system in terms of $ \Pi_i$ is
$$ \dot{\Pi}_i=\widehat{\Pi}_i \sum_{j\in \mathcal{N}_i} h({\theta}_{ij})\widehat{\Pi}_i\Pi_j, $$
for all $i\in \mathcal{V}$.
\end{lemma}
\rev{
In what follows, we denote $\mathbf{\Gamma}=(\Gamma_1^T, \dots, \Gamma_{N_0}^T)^T$, then the formation of regular polyhedron $\{p,q\}$ in $\mathcal{S}^2$ is defined by
\[
 	\mathcal{M}_{\{p,q\}}\! =\! \left\{\mathbf{\Gamma} \!\in\! (\mathcal{S}^2)^{N_0} \!:\! \mathbf{\Gamma}\!=\!(\mathbf{I}_{N_0} \!\otimes R) \mathbf{\Gamma}^{\{p,q\}}, \forall R \!\in\! \mathcal{SO}(3)\right\},
\]
where $\otimes$ is the Kronecker product and  $\mathbf{\Gamma}^{\{p,q\}} \in (\mathcal{S}^2)^{N_0}$ is a given state defining a formation of regular polyhedron $\{p,q\}$.
Although there is no exact expression of $\mathbf{\Gamma}^{\{p,q\}}$ given here for each $\{p,q\}$, in the next section based on the symmetries of regular polyhedra, another expression of $\mathcal{M}_{\{p,q\}}$ is provided, by which we are able to handle all platonic solids within a unified framework.
}

\rev{Denote $\mathcal{G}_{\{p,q\}}$ as the inter-agent topology employed for formation $\{p,q\}$.} Now, we are ready to pose the intrinsic formation problem investigated in this paper.
\begin{problem}\label{prob:Int:Problem1}
   In closed-loop system (\ref{eq:PF:ClosedSystem}), for all integer $p$, $q$ satisfying inequality (\ref{eq:pre:relation_p_q}), find a proper inter-agent graph $\mathcal{G}_{\{p,q\}}$ with $N_0^{\{p,q\}}$ vertices, such that the regular polyhedra formation $\mathcal{M}_{\{p,q\}}$ is invariant and further asymptotically stable.
\end{problem}
Since in the intrinsic formation scheme, the control protocol only contains some simple interaction, for example a repulsion in \eqref{eq:PF:Control}, and the desired pattern is constructed based on the designed connection topology,
in the following section we give some design criteria for finding candidate
graphs that can solve Problem \ref{prob:Int:Problem1}.

\section{Design of Graph}\label{sec:graph}
In this section, the inter-agent topology is designed based on the
symmetry properties possessed by the Platonic solids. Under a symmetry assumption on the connection, we give a family of possible graphs that can make the desired formations invariant in the closed-loop system.
\subsection{Symmetries of Platonic Solids}\label{sec:symmetryPlatonicSolid}
First, we give a coordinate-based description for the symmetries of regular polyhedra.

For a regular polyhedron $\{p,q\}$, each rotational symmetry can be identified by a pair $(R,\sigma)$, in which $R$ is a rotation about some axis passing the center of $\{p,q\}$, and $\sigma$ is a permutation among vertices acting equivalently as rotation $R$. We denote $\mathcal{H}_{\{p,q\}}=\big\{(R_i,\sigma_i)\big\}_{i \in\mathcal{V}}$ as a subset of all rotational symmetries, where the rotation map $R_i: (\mathcal{S}^2)^{N_0}\to \mathcal{SO}(3)$ defined by $R_i(\mathbf{\Gamma})=\exp(\frac{2\pi}{q}\widehat{\Gamma}_i)$, and $\sigma_i$ is the permutation acting equivalently with $R_i$ when $\mathbf{\Gamma}=\mathbf{\Gamma}^{\{p,q\}}$.

Therefore, we obtain a description of vertices set for the regular polyhedron $\{p,q\}$ as
\begin{align}\label{eq:Int:equilibriaSet2}
   \mathcal{M}'_{\{p,q\}}\!\!&=\!\!\Big\{\mathbf{\Gamma}  \!\in \! (\mathcal{S}^2)^{N_0}  :  \; \exists m \neq n \in \mathcal{V}, \; s.t.  \; \widehat{\Gamma}_m \Gamma_n\neq 0 ; \nonumber \\
        &\!\!\!\!\big(\mathbf{I}_{N_0} \!\otimes \!R(\mathbf{\Gamma}) \!- \!P_{\sigma} \!\otimes\! \mathbf{I}_3 \big)\mathbf{\Gamma}\!=\!0,\, \forall (R,\!\sigma) \!\!\in\! \mathcal{H}_{\{p,q\}}\!\!\Big\},
\end{align}
where the condition $\widehat{\Gamma}_m \Gamma_n\neq 0$ is to eliminate the consensus and antipodal configurations of all vertices.
Illustratively, the permutations in $\mathcal{H}_{\{3,3\}}$, for example, are $\sigma_1=(1)(2,3,4)$, $\sigma_2=(2)(1,4,3)$, $\sigma_3=(3)(1,2,4)$, $\sigma_4=(4)(2,1,3)$, where the cycle notation is used.

\rev{
Actually it can be shown that two representations of polyhedral formations are identical, i.e., $\mathcal{M}'_{\{p,q\}}=\mathcal{M}_{\{p,q\}}$.
}
\rev{
\begin{prop}\label{thm:Int:twoMEquivalent}
For all integer $p$, $q$ satisfying inequality (\ref{eq:pre:relation_p_q}), $\mathcal{M}'_{\{p,q\}}=\mathcal{M}_{\{p,q\}}$.
\end{prop}
}
\begin{proof}
\rev{See the proof in Appendix~\ref{sec:AppendixPfM}.}
\end{proof}

\rev{
Due to Proposition~\ref{thm:Int:twoMEquivalent},  in the rest of the paper we omit the prime and use $\mathcal{M}_{\{p,q\}}$ presenting the regular polyhedron formation $\{p, q\}$ defined in \eqref{eq:Int:equilibriaSet2}.
}

\subsection{Symmetries of Inter-agent Topology}
\rev{
Driven by the simple antagonistic interaction
(\ref{eq:PF:ClosedSystem}), the construction of desired formations
depends heavily on the inter-agent graph employed. Since the five
Platonic solids possess the most symmetries in all polyhedra,
intuitively some symmetries should also be inherited by the designed graph.
}

In order to characterize symmetries of a graph, we give the definition of graph automorphism.
\begin{definition}
  For a graph $\mathcal{G}=(\mathcal{V},\mathcal{E})$, a permutation
  specified by mapping $\sigma: \mathcal{V} \rightarrow \mathcal{V}$
  is a graph automorphism, if the edge set satisfies $(\sigma(i),\sigma(j))\in
  \mathcal{E}$ if and only if $(i,j)\in \mathcal{E}$.
\end{definition}
For a graph automorphism, the following remark gives an alternative characterization that is easy to check.
\begin{remark}\label{rem:GD:PACommute}
  Let $A$ be the adjacency matrix of graph $\mathcal{G}$, then a permutation $\sigma$ is an automorphism of $\mathcal{G}$, if and only if $AP_\sigma=P_\sigma A$, where  $P_\sigma$ is the permutation matrix of $\sigma$.
\end{remark}
We use the next assumption to indicate that the inter-agent topology
designed
will share the same symmetries with the corresponding polyhedron.
\begin{assumption}[graph symmetry] \label{ass:GD:graphSymmetry}
  The inter-agent graph $\mathcal{G}_{\{p,q\}}$ is undirected, connected, and each permutation in $\mathcal{H}_{\{p,q\}}$ is an automorphism of this graph.
\end{assumption}
In the following, it will be shown that under a graph satisfying Assumption~\ref{ass:GD:graphSymmetry} the manifold consisting of the desired polyhedra formations is invariant. To this end, we start with the next lemma.
\begin{lemma}\label{thm:GD:symmetricDynamics}
 Let a vector field be $f:\mathbb{R}^n \times \mathbb{R}^+ \rightarrow \mathbb{R}^n$. Then the collection of invertible linear transformations,
 $$\Pi\!=\!\big\{A \!\in\! GL(n, \mathbb{R})\!:\! f(Ax,t)\!=\!Af(x,t), \forall \! x,t\big\}$$
 constructs a group under matrix multiplication. Moreover, for any $A \in \Pi $, the set $M_A=\{x \in \mathbb{R}^n: A x=x\}$ is invariant under dynamics $\dot{x}=f(x,t)$.
 \end{lemma}
\begin{proof}
Firstly, we prove $\Pi$ forms a group under the matrix multiplication. \textbf{Closure}: Let $A, B \in \mathbb{R}^{n \times n}$ such that $f(Ax,t)=Af(x,t)$, and $f(Bx,t)=Bf(x,t)$, then $f(ABx,t)=ABf(x,t)$. \textbf{Associativity}: Follows from the associativity of matrix multiplication. \textbf{Identity }: $\mathbf{I}_n \in \Sigma$ as $f(\mathbf{I}_nx,t)=\mathbf{I}_nf(x,t)$. \textbf{Inverse }:  For any $A \in \Pi$ and $x \in \mathbb{R}^n$, let $y=A^{-1}x$, then $f(x)=f(Ay)=Af(y)$. This implies $f(A^{-1}x)=f(y)=A^{-1}f(x)$. Thus $A^{-1} \in \Pi$. Furthermore, if $A \in \Pi$, for any $x_0 \in M_A$, we have $f(x_0)=f(A x_0)=A f(x_0)$. Thus $f(x_0) \in M_A$. Since the tangent space of $M_A$ at $x_0$ satisfies $ T_{x_0}M_A=M_A$, we obtain the invariance of set $M_A$ under dynamics $\dot{x}=f(x,t)$.
\end{proof}

Then with help of Lemma~\ref{thm:GD:symmetricDynamics} we have the following theorem.
\begin{theorem}\label{thm:GD:invarianceOfFormation}
  Under Assumption~\ref{ass:GD:graphSymmetry}, the regular polyhedra formation $\mathcal{M}_{\{p,q\}}$ is invariant in closed-loop system (\ref{eq:PF:ClosedSystem}).
\end{theorem}
\begin{proof}
  The function $\widehat{\Gamma}_i^2\Gamma_j$ is continuously differentiable and hence Lipschitz. The control gain $h(\theta_{ij})=\bar{h}(\Gamma_i^T\Gamma_j)$ is also Lipschitz in $\mathbf{\Gamma}$.
  In addition, the configuration manifold $(\mathcal{S}^2)^{N_0}$ is compact. 
  As such, dynamics (\ref{eq:PF:ClosedSystem}), with initial
  condition $\mathbf{\Gamma}(0)=\mathbf{\Gamma}^\ast$, has a unique
  solution $\mathbf{\Gamma}(\mathbf{\Gamma}^\ast,t)$, for $t \in
  [0,\infty)$. It can be shown by straight forward computation that the set $\Omega_0=\{\mathbf{\Gamma} \in (\mathcal{S}^2)^{N_0}  :  \widehat{\Gamma}_m\Gamma_n =0, \forall m,n \in \mathcal{V}\}$ is invariant under system (\ref{eq:PF:ClosedSystem}). Due to the uniqueness of solution $\mathbf{\Gamma}(\mathbf{\Gamma}^\ast,t)$, its complementary set $\Omega_0^C$ is also invariant.

  We denote the closed-loop system as
  $\dot{\mathbf{\Gamma}}=F(\mathbf{\Gamma})$,
  where $F(\cdot)$ is the stacked form of (\ref{eq:PF:ClosedSystem}). Due to Lemma~\ref{Thm:PF:rotationalInvariant}, for any $R \in \mathcal{SO}(3)$ and $\mathbf{\Gamma} \in  (\mathcal{S}^2)^{N_0}$,  we have
  \begin{equation}\label{eq:GD:rotaionInvariance}
    F[(\mathbf{I}_{N_0}\otimes R) \mathbf{\Gamma}] = (\mathbf{I}_{N_0}\otimes R) F(\mathbf{\Gamma}).
  \end{equation}

  In addition, we denote $A_\mathcal{G}=\big[a_{ij}\big]_{i,j \in
    \mathcal{V}}$ as the adjacency matrix of the inter-agent graph $\mathcal{G}$. Then for any $(\overline{R}, \overline{\sigma}) \in \mathcal{H}_{\{p,q\}}$,  by the individual closed-loop dynamics (\ref{eq:PF:ClosedSystem}),
  \begin{eqnarray}
    F[(P_{\overline{\sigma}} \otimes \mathbf{I}_3) \mathbf{\Gamma}]=
    \begin{bmatrix}
      \sum\limits_{j \in \mathcal{V}}{a_{1j}f(\Gamma_{\overline{\sigma}(1)},\Gamma_{\overline{\sigma}(j)})}\\
      \vdots\\
      \sum\limits_{j \in \mathcal{V}}{a_{N_0j}f(\Gamma_{\overline{\sigma}(N_0)},\Gamma_{\overline{\sigma}(j)})}\\
    \end{bmatrix}, \label{eq:GD:permutationInvariance_1}
  \end{eqnarray}
  \vspace{-2em}
  \begin{eqnarray}
    (P_{\overline{\sigma}} \otimes \mathbf{I}_3)F(\mathbf{\Gamma})=
    \begin{bmatrix}
      \sum\limits_{j \in \mathcal{V}}{a_{\overline{\sigma}(1)j}f(\Gamma_{\overline{\sigma}(1)},\Gamma_{j})}\\
      \vdots\\
      \sum\limits_{j \in \mathcal{V}}{a_{\overline{\sigma}(N_0)j}f(\Gamma_{\overline{\sigma}(N_0)},\Gamma_{j})}\\
    \end{bmatrix},\label{eq:GD:permutationInvariance_2}
  \end{eqnarray}
   where
   $f(\Gamma_i,\Gamma_j)=h(\theta_{ij})\widehat{\Gamma}_i^2\Gamma_j$. Since
   graph $\mathcal{G}$ satisfies
   Assumption~\ref{ass:GD:graphSymmetry}, we have
   $a_{ij}=a_{\overline{\sigma}(i)\overline{\sigma}(j)}$ which implies
   that (\ref{eq:GD:permutationInvariance_1}) and (\ref{eq:GD:permutationInvariance_2}) are equal. Combining this fact with (\ref{eq:GD:rotaionInvariance}), Lemma~\ref{thm:GD:symmetricDynamics} gives that $T=(P_{\overline{\sigma}} \otimes \mathbf{I}_3)^{-1}(\mathbf{I}_{N_0}\otimes R)$ is also an invariant transformation, i.e., $F(T\mathbf{\Gamma})=TF(\mathbf{\Gamma})$. Moreover, set $\{\mathbf{\Gamma} \in  (\mathcal{S}^2)^{N_0}:T\mathbf{\Gamma}=\mathbf{\Gamma}\}$ is invariant in closed-loop system.

Therefore, for any $(R_i,\sigma_i) \in \mathcal{H}_{\{p,q\}}$, where $i\in \mathcal{V}=\big\{1,\cdots,N_0^{\{p,q\}}\big\}$, we have $\Omega_i=\big\{\mathbf{\Gamma} \in (\mathcal{S}^2)^{N_0}  : (\mathbf{I}_{N_0} \!\otimes \!R_i - P_{\sigma_i} \!\otimes\! \mathbf{I}_3 )\mathbf{\Gamma}\!=\!0\big\}$ is invariant. Since $\mathcal{M}_{\{p,q\}}$ admits the composition $\mathcal{M}_{\{p,q\}}=\Omega_0^C \cap \bigcap\limits_{i \in \mathcal{V}}\Omega_i$, the assertion follows.
\end{proof}

\rev{
\begin{remark}\label{rm:GD:noGlobalController}
From the proof of Theorem~\ref{thm:GD:invarianceOfFormation}, we can see that the consensus and antipodal configurations are always equilibria of the closed-loop system (\ref{eq:PF:ClosedSystem}). Thus the global asymptotic stability of any formations  for reduced attitude are not even possible. This is also implied by the fact that there is no continuously differentiable global stabilizer for the dynamical systems on $\mathcal{S}^2$, since $\mathcal{S}^2$ is a closed manifold without boundary \cite{Bhat00scl}.
\end{remark}
}

\subsection{Possible Inter-agent Topology}
For the formation $\{p,q\}$, by virtue of Remark~\ref{rem:GD:PACommute} 
all graphs fulfilling Assumption~\ref{ass:GD:graphSymmetry} can be specified.

We denote the complete graph with $N$ vertices by $\mathcal{K}_{N}$. And a \textit{Platonic graph}, denoted by ${\mathcal{P}}_{\{p,q\}}$, is referred to as an undirected graph admitting the skeleton of Platonic solid $\{p,q\}$ as its edges.
Note that $\mathcal{P}_{\{3,3\}}=\mathcal{K}_{4}$.
Since any permutation is an automorphism of a complete graph and each permutation in $\mathcal{H}_{\{p,q\}}$ must be an automorphism of Platonic graph $\mathcal{P}_{\{p,q\}}$, we have two trivial graphs
$\mathcal{K}_{N_0^{\{p,q\}}}$ and $\mathcal{P}_{\{p,q\}}$ satisfying Assumption~\ref{ass:GD:graphSymmetry}.
We also present all other possible graphs in Appendix~\ref{sec:AppendixGraph}.

In the next section, the method for investigating the exponential stability of $\mathcal{M}_{\{p,q\}}$ in the closed-loop system under these graphs will be discussed.
\rev{Due to Remark~\ref{rm:GD:noGlobalController}, although global stability of the desired formations is more desirable, but it is actually inaccessible and the best result regarding stability of systems in $\mathcal{S}^2$ is the so-called almost-global stability, which requires to  exactly characterize all equilibria for the nonlinear closed-loop system \cite{markdahl2017almost}. In \cite{zhang2017tetrahedron} all equilibrium configurations for the regular tetrahedron case have been investigated, but unfortunately it fails here since solving systems of nonlinear multivariable equations becomes intractable when $N_0$ is larger.
}

%

\section{Stability Analysis of Desired Formations}\label{sec:stability}

In this section, by a novel coordinates transformation in
$\mathcal{S}^2$, we show first that stability of the desired formation
is equivalent to stability of a constrained nonlinear system with a
higher dimension. Then, a method for investigating stability of
constrained systems is provided. Furthermore, to avoid the difficulty
in eliminating redundant constraints introduced, we show that it is sufficient to study a simplified system with much less constraints.

\subsection{Coordinates Transformation}

\rev{
We set a new coordinates system consisting of the relative attitudes between any two agents $i$, $j$ and the absolute attitude of the whole formation. For every $\mathbf{\Gamma}\in (\mathcal{S}^2)^{N_0}$, the coordinates transformation is denoted by $\xi=[\xi_s^T,\xi_c^T]^T=\Phi(\mathbf{\Gamma})$.
In coordinates $\xi$, the component $\xi_s$ represents the relative attitude, and is defined by
\begin{equation}\label{eq:Method:ComponentsXi_s}
\xi_s\!=\![\xi_{12},\xi_{13},\!\cdots\!, \xi_{1N_0},\xi_{23}, \!\cdots \!, \xi_{2N_0}, \!\cdots\!, \xi_{N_0-1,N_0}]^T\!,
\end{equation}
in which $\xi_{ij}=\Gamma_i^T\Gamma_j$. The absolute attitude component $\xi_c=[\phi_1,\psi_1,\gamma]^T$, where $\gamma=\mathrm{atan2}( \cos(\phi_2\!)\sin(\psi_2\!-\!\psi_1\!),\sin(\phi_1\!)\cos(\phi_2\!)\cos(\psi_2\!\!-\!\!\psi_1\!)\!-\!\cos(\phi_1\!)\sin(\phi_2\!))
$
and $(\phi_k, \psi_k)$ are RPY angles of $\Gamma_{k}$. By this definition, component $\xi_c$ specifies the attitude of $\Gamma_1, \Gamma_2$ or equivalently the whole formation relative to the inertial frame $\emph{O-XYZ}$. The details on the meaning of $\xi_c$ can be found in \cite{zhang2017tetrahedron}.
}

Then by the above transformation, the closed-loop dynamics (\ref{eq:PF:ClosedSystem}) becomes
\begin{subequations}\label{eq:SF:newCorDynf}
    \begin{equation}\label{eq:SF:newCorDynf1}
      \dot{\xi}_s= \hat{f}_s(\xi_s),
    \end{equation}
    \begin{equation}\label{eq:SF:newCorDynf2}
        \dot{\xi}_c= \hat{f}_c(\xi_c,\xi_s).
    \end{equation}
\end{subequations}
Due to Lemma~\ref{Thm:PR:cosine}, after some involved algebraic manipulation, the elements of $\hat{f}_s(\cdot)$ in (\ref{eq:SF:newCorDynf1}) can be derived as
\begin{equation*}
  \dot{\xi}_{ij}=\!\sum\limits_{k \in \mathcal{N}_j}h(\theta_{jk})\left(\xi_{ij}\xi_{jk}-\xi_{ik}\right) \! +\!\sum\limits_{k \in \mathcal{N}_i}h(\theta_{ik})\left(\xi_{ij}\xi_{ik}-\xi_{jk}\right).
\end{equation*}

With the help of the coordinates introduced, the closed-loop dynamics achieves a triangular form, namely, the dynamics of $\xi_s$ only depends on $\xi_s$ itself. Moreover, under the new coordinates, the stability of manifold $\mathcal{M}_{\{p,q\}}$ is equivalent to the stability of one corresponding equilibrium $\xi_s= \xi_s^{\{p,q\}}$ in subsystem  (\ref{eq:SF:newCorDynf1}). \rev{For example, for regular tetrahedron $\mathcal{M}_{\{3,3\}}$ the corresponding equilibrium is $\xi_s^{\{3,3\}}=[\frac{1}{3},\frac{1}{3},\frac{1}{3},\frac{1}{3},\frac{1}{3},\frac{1}{3}]^T$.}

We note that in the coordinates $\xi$, the number of variables is in total
$D_\xi = N_0(N_0-1)/2+3$,
which are far more than the degrees of freedom, $2N_0$, of $\mathbf{\Gamma}\in (\mathcal{S}^2)^{N_0}$. This is because in space $(\mathcal{S}^2)^{N_0}$ there exist inherent constraints for the elements of $\xi_s$. We state these constraints in the following lemma.

\begin{lemma}\label{Thm:Method:IdentityFor4P}
For any $4$-element set ${C}=\{i,j,k,l\} \subset \mathcal{V}$, if $\Gamma_i,\Gamma_j,\Gamma_k,\Gamma_l \in \mathcal{S}^2$, we have the following identity (based on Problem 70 in \cite{dorrie2013100Problem}):
\begin{gather}\label{eq:Method:Constraint}
   g_{C}(\xi_s)
  :=\det
    \begin{pmatrix}
      1 & \xi_{ij}& \xi_{ik}& \xi_{il} \\
       \xi_{ij}& 1 & \xi_{jk}& \xi_{jl} \\
       \xi_{ik}&  \xi_{jk}& 1 &\xi_{kl} \\
       \xi_{il}&  \xi_{jl} &\xi_{kl} & 1\\
    \end{pmatrix}
    =0,
  \end{gather}
 where $\det(\cdot)$ is the determinant of a matrix and $\xi_{ij}=\Gamma_i^T\Gamma_j$.
\end{lemma}

We conclude the above discussion with the following remark.
\begin{remark}\label{Rmk:SF:maniToPoint1}
Stability of desired formation $\mathcal{M}_{\{p,q\}}$ in closed-loop system (\ref{eq:PF:ClosedSystem}) is equivalent to stability of equilibrium $\xi_s= \xi_s^{\{p,q\}}$ in subsystem (\ref{eq:SF:newCorDynf1}) under the constraints $g_{C}(\xi_s)=0$, for any 4-element set $C \subset \mathcal{V}$.
\end{remark}

In the rest of the paper, denote the manifold $\mathcal{M}_C$ as
\begin{equation}\label{eq:Method:constraintManifold}
  \mathcal{M}_C\!=\!\Big\{ \xi_s\!: g_{C}(\xi_s)\!=\!0,\; \forall\text{4-element set}\, C \!\subset \!\mathcal{V} \!\Big\}.
\end{equation}
\rev{
We note that for a polyhedron with $N_0$ vertices, the number of constraints in \eqref{eq:Method:constraintManifold} is $\binom{N_0}{4}$,  which becomes enormous when $N_0$ is large. But actually only $m_0$ of them is needed to cast out the variable redundancy in the new coordinate $\xi$, where
\begin{equation}\label{eq:SF:DegreeRuleOut}
m_0=\frac{(N_0-2)(N_0-3)}{2}.
\end{equation}
}

In the next section in order to investigate a system subject to algebraic constraints we appeal to the concept stability of a system restricted to a manifold.


\subsection{Stability of System Restricted to Manifold}
Firstly, we consider the case of a linear system, which is defined by
\begin{equation}\label{LIS}
  \dot{x}=A\rev{x},\;\;\; x\in \mathbb{R}^n.
\end{equation}
Then we give the definition of \rev{asymptotic} stability for system (\ref{LIS}) restricted to a subspace.
\begin{definition}\label{def:stableRestrictLinear}
  Let $F \in \mathbb{R}^{m \times n}$ be full row rank, and $\mathbf{V}=\big\{x \in \mathbb{R}^n: F\rev{x}=0\big\}$ be a $(n-m)$-dimensional subspace in $\mathbb{R}^n$. We say that system (\ref{LIS}) \emph{restricted to $\mathbf{V}$\, is \rev{asymptotically} stable}, if for any trajectory $x^*(t)$ satisfying  $x^*(t) \in \mathbf{V},  \forall t \in [0,+\infty)$,  $x^*(t) \rightarrow 0$ as $t \rightarrow \infty$.
\end{definition}
We note that this definition is well-posed in the sense that  at least $x^*(t)=0$ is a trajectory always located in $\mathbf{V}$.
\rev{
\begin{remark}
Actually stability in Definition~\ref{def:stableRestrictLinear} is weaker than the condition that trajectory $x^*(t) \to 0$, $\forall x^*(0)\in \mathbf V$, i.e., the dynamics corresponding to subsapce $\mathbf V$  is stable.
\end{remark}
}
Let $\mathbf{V}^*$ be the maximal A-invariant subspace \cite[Sec. 0.7]{wonham1974linear} in $\mathbf{V}$. Since any trajectory $x(t) \in \mathcal{V}, \forall t \geq 0$ has to evolve in $\mathbf{V}^*$, an equivalent condition is provided in the next proposition.
\begin{proposition}[Equivalent Condition]\label{Thm:iff_SRS}
  Let $F \in \mathbb{R}^{m \times n}$ be full row rank, and $\mathbf{V}=\big\{x \in \mathbb{R}^n: Fx=0\big\}$. The system (\ref{LIS}) restricted to $\mathbf{V}$ is \rev{asymptotically} stable, if and only if the dynamics corresponding to subspace $\mathbf{V}^*$ is \rev{asymptotically} stable, where $\mathbf{V}^*$ is the maximal A-invariant subspace in $\mathbf{V}$.
\end{proposition}
Then in order to avoid the complexity to obtain the inverse of matrices we introduce an orthonormal coordinates transformation as
\begin{align}\label{3PartCoordinates}
  T&=\begin{bmatrix}  T_1\\ T_2 \\ T_3\end{bmatrix}  = \begin{bmatrix}  T_1^T \vrule & T_2^T \vrule & T_3^T\end{bmatrix}^T \nonumber \\
  &=\begin{bmatrix}  v_1 \! & \!  \cdots \!  & \! v_r \vrule \! & \! v_{r+1} \!&\!  \cdots \! &\! v_{n-m} \vrule \!&\! v_{n-m+1} \!&\!  \cdots \!&\! v_n\end{bmatrix}^T,
\end{align}
where $\left\{v_1, \cdots, v_n\right\}$ is an orthonormal basis of
$\mathbb{R}^n$, $\left\{v_1, \cdots, v_r\right\}$ is a basis of the maximal A-invariant subspace $\mathbf{V}^*$, and $\left\{v_{n-m+1}, \cdots,  v_n \right\}$ is a basis of $\mathbf{I\!m}(F^T)$. Then we have $T^{-1}=T^T$. Moreover, by the Gram-Schmidt process, there is an invertible matrix $O_3 \in \mathbb{R}^{m \times m}$ orthonormalizing the rows of $F$ as $T_3= O_3F$.

Let $x=\big[T_1^T,T_2^T,T_3^T\big] z$, where $z=[z_1^T,z_2^T,z_3^T]^T$ with a compatible partition. Since $\mathbf{V}^*$ is an A-invariant subspace, we have $T_2AT_1^T=0$ and $T_3AT_1^T=0$.  Specifically, the dynamics of $z$ reads
  \begin{align}\label{3PartTrans1}
    \begin{bmatrix} \dot{z}_1\\ \dot{z}_2 \\ \dot{z}_3 \end{bmatrix}
     = \begin{bmatrix} T_1 A T_1^T & T_1 A T_2^T & T_1 A T_3^T \\ 0 & T_2 A T_2^T & T_2 A T_3^T \\ 0 & T_3 A T_2^T & T_3 A T_3^T \end{bmatrix} \begin{bmatrix}
     z_1\\ z_2 \\ z_3 \end{bmatrix}.
  \end{align}

\begin{lemma}\label{thm:controlability}
   Let $F \in \mathbb{R}^{m \times n}$ be full row rank, and for system (\ref{LIS}) $\mathbf{V}^*$ be the maximal A-invariant subspace in $\mathbf{V}=\big\{x \in \mathbb{R}^n: Fx=0\big\}$, then under coordinates transformation (\ref{3PartCoordinates}), the pair $(T_3AT_2^T, T_2AT_2^T)$ is observable.
\end{lemma}
\begin{proof}
\rev{See the proof in Appendix~\ref{sec:AppendixPfControllability}.}

\end{proof}

In what follows, unless otherwise mentioned, we assume $F \in
\mathbb{R}^{m \times n}$ is a full row rank matrix. $F$ can be
partitioned into $[F_1,F_2]$, where $F_1 \in \mathbb{R}^{m \times
  (n-m)}$ and  $F_2 \in \mathbb{R}^{m \times m}$. Without loss of
generality, we suppose $F_2$ is invertible. Since $F$ is a full row
rank matrix, this assumption would always hold by some rearrangement
among the variables. The next theorem offers the possibility to extend Proposition~\ref{Thm:iff_SRS} to nonlinear systems.
\begin{theorem}\label{Thm:Method:restrictStableNSCond}
  Let $F \in \mathbb{R}^{m \times n}$, $\mathbf{V}=\big\{x \in \mathbb{R}^n: Fx=0\big\}$. We assume $F_2$ is invertible. Then system (\ref{LIS}) restricted to $\mathbf{V}$ is \rev{asymptotically} stable, if and only if there exists $P \in \mathbb{R}^{(n-m) \times n}$ such that
  \begin{enumerate}[label=(\alph*)]
    \item $\begin{bmatrix} P\\ F \end{bmatrix}$ is invertible.\\
    \item $PA \begin{bmatrix} I_{n-m}\\ -F_2^{-1}F_1\end{bmatrix} \left( P\begin{bmatrix} I_{n-m}\\ -F_2^{-1}F_1 \end{bmatrix} \right)^{-1}$  is stable.
  \end{enumerate}
\end{theorem}

\begin{proof}
Necessity:
   By Lemma~\ref{thm:controlability}, there is a matrix $K \in \mathbb{R}^{m \times (n-m-r)}$ such that $T_2AT_2^T+ K T_3AT_2^T $ is stable. Let $Q=\begin{bmatrix}  I_r & & \\ & I_{(n-m-r)} & K \\ & & I_m \end{bmatrix}$.
  Then we introduce a new coordinates as $y=QTx$.
  The dynamics in the new coordinates can be obtained as
   \begin{align}\label{3PartTrans2}
    \begin{bmatrix} \dot{y}_1\\ \dot{y}_2 \\ \dot{y}_3 \end{bmatrix} &= \begin{pmat}[{.|}] T_1 A T_1^T & T_1 A T_2^T   & * \cr 0 & T_2 A T_2^T + K T_3 A T_2^T &* \cr\- 0 & T_3 A T_2^T & * \cr \end{pmat} \begin{bmatrix}   y_1\\ y_2 \\ y_3 \end{bmatrix} \nonumber \\
    &= \begin{pmat}[{|}] \widetilde{A}_1 &\widetilde{A}_2 \cr\- \widetilde{A}_3 & \widetilde{A}_4 \cr \end{pmat}
    \begin{bmatrix}
     y_1\\ y_2 \\ y_3 \end{bmatrix}.
   \end{align}
   Furthermore, due to the asymptotic stability of system (\ref{LIS}) restricted to $\mathbf{V}$, Proposition~\ref{Thm:iff_SRS} gives that $T_1AT_1^T$ is Hurwitz. Thus, $\widetilde{A}_1$ is a stable matrix.

   Let $P=\begin{bmatrix} T_1 \\ T_2+K O_3F \end{bmatrix}$, then we show that $P$ satisfies conditions (a), (b) in the theorem.  Denote $\widetilde{F}=O_3F=\big[O_3F_1, \; O_3F_2\big]$, and $P=\begin{bmatrix} P_1 & P_2\end{bmatrix}$ where $P_1 \in \mathbb{R}^{(n-m) \times (n-m)}$, and $P_2 \in \mathbb{R}^{(n-m) \times m}$.
   Since $O_3F_2$ is invertible, the dynamics turns to (\ref{3PartTrans3}) (located at the top of the next page).

\begin{figure*}[tp]
\normalsize
\begin{equation}\label{3PartTrans3}
    \begin{bmatrix} \dot{y}_1\\ \dot{y}_2 \\ \dot{y}_3 \end{bmatrix} \!=\! \begin{bmatrix} P_1 & P_2 \\ O_3F_1 & O_3F_2 \end{bmatrix} A \begin{bmatrix}  \big(P_1\!-\!P_2F_2^{-1}F_1\big)^{-1} & -\big(P_1\!-\!P_2F_2^{-1}F_1\big)^{-1}P_2F_2^{-1}O_3^{-1} \\ -F_2^{-1}F_1\big(P_1\!-\!P_2F_2^{-1}F_1\big)^{-1}  & F_2^{-1}O_3^{-1}\!+\!F_2^{-1}F_1\big(P_1\!-\!P_2F_2^{-1}F_1\big)^{-1}P_2F_2^{-1}O_3^{-1}\end{bmatrix}
    \begin{bmatrix} y_1 \\ y_2\\ y_3  \end{bmatrix}.
    \end{equation}
\hrulefill
\end{figure*}

    Comparing (\ref{3PartTrans3}) with (\ref{3PartTrans2}), we obtain  the matrix
    \begin{eqnarray*}
      \widetilde{A}_1
      =PA\begin{bmatrix} I_{n-m}\\ -F_2^{-1}F_1\end{bmatrix} \left( \begin{bmatrix} P_1 & P_2 \end{bmatrix} \begin{bmatrix} I_{n-m}\\ -F_2^{-1}F_1 \end{bmatrix} \right)^{-1}
    \end{eqnarray*}
    As $\widetilde{A}_1$ is Hurwitz, we have (b) hold. On the other hand, it is obvious that $\begin{bmatrix} P\\ F \end{bmatrix}$ is invertible.

    Sufficiency: we only need to show $y(t):=Px(t) \rightarrow 0$, as $t\rightarrow 0$, which is omitted here.
\end{proof}

Now we are ready to extend asymptotic stability of linear systems restricted to a subspace to nonlinear systems. Let $\mathcal{M}=\big\{x \in \mathbb{R}^n: G(x)=0\big\}$ be a manifold, where $G:\mathbb{R}^n \rightarrow \mathbb{R}^m $ and $\left.\frac{\partial G}{\partial x}\right|_{x=x_e}$ has full row rank. We consider exponential stability restricted to $\mathcal{M}$ for a nonlinear system
\begin{equation}\label{NLS}
  \dot{x}=f(x),\;\;\; x\in \mathbb{R}^n.
\end{equation}

\begin{definition}
  Suppose $x_e\in \mathcal{M}$ is an equilibrium of system (\ref{NLS}), and $\left.\frac{\partial G}{\partial x}\right|_{x=x_e}$ has full row rank. We say system (\ref{NLS}) restricted to $\mathcal{M}$ is exponentially stable at $x_e$, if there is a neighbourhood of $x_e$, denoted by $B(x_e)$, such that for any trajectory $x^*(t)$ of (\ref{NLS}) satisfying $x^*(0) \in B(x_e) \cap \mathcal{M}$ and $x^*(t) \in \mathcal{M}$, $\forall t\geq 0$, there are positive constants $\alpha, \beta$ such that $\|x^*(t)-x_e\| \leq \alpha e^{-\beta t}$, $\forall t  \geq 0$.
\end{definition}

\rev{ The above definition is well defined since we always have trivial solution $x^*(t)=x_e, \forall t \geq 0$.}
Based on Theorem~\ref{Thm:Method:restrictStableNSCond}, the next theorem provides a method to investigate the stability of a nonlinear system restricted to manifold $\mathcal{M}$.
\begin{theorem}\label{Thm:Method:NonlinearRestrict}
  The nonlinear system (\ref{NLS}) restricted to manifold $\mathcal{M}$ is exponentially stable at $x_e$, if the linearized system
    $\dot{\bar{x}}= \bar{A} \bar{x}= \frac{\partial f(x_e)}{\partial x} \bar{x}$
  restricted to subspace $\bar{\mathbf{V}}$ is stable, where $\bar{x}=x-x_e$ and $\mathbf{\bar{V}}=\big\{\bar{x} \in \mathbb{R}^n : \left.\frac{\partial G}{\partial x}\right|_{x=x_e} \bar{x} =0 \big\}$.
\end{theorem}

\begin{proof}
  We denote $\bar{F}=\left.\frac{\partial G}{\partial
      x}\right|_{x=x_e}=[\bar{F}_1,\bar{F}_2]$, where $\bar{F}_2 \in
  \mathbb{R}^{m \times m}$. Without loss of generality, we suppose
  $\bar{F}_2$ is invertible, otherwise some rearrangement of the variables is needed. Since $\dot{\bar{x}}= \bar{A} \bar{x}$ restricted to $\bar{\mathbf{V}}$ is stable, by Theorem~\ref{Thm:Method:restrictStableNSCond}, there is a matrix $P \in \mathbb{R}^{(n-m)\times n}$ such that
  \begin{eqnarray}
    &&\begin{bmatrix} P\\ \bar{F} \end{bmatrix} \text{is invertible}. \label{eq:Method:stable4NLS_Con1} \\
    &&P\bar{A} \begin{bmatrix} I_{n-m}\\ -\bar{F}_2^{-1}\bar{F}_1\end{bmatrix} \left( P\begin{bmatrix} I_{n-m}\\ -\bar{F}_2^{-1}\bar{F}_1 \end{bmatrix} \right)^{-1} \text{is stable}.\label{eq:Method:stable4NLS_Con2}
  \end{eqnarray}
  Then we set a transformation
  \begin{equation}\label{eq:Method:stable4NLS_Proof1}
  \varphi=\begin{bmatrix} \varphi_\Romannum{1}\\ \varphi_\Romannum{2} \end{bmatrix}=\Psi(\bar{x})=\begin{bmatrix} P\bar{x}\\ G(\bar{x}+x_e) \end{bmatrix}.
  \end{equation}
  By (\ref{eq:Method:stable4NLS_Con1}), we have $\left.\frac{\partial
      \Psi}{\partial \bar{x}}\right|_{\bar{x}=0}$ is invertible, thus
  this transformation is a diffeomorphism in the neighborhood of the
  origin. Moreover, by the inverse function theorem, there exist an inverse mapping $\Psi^{-1}: \mathbb{R}^n \rightarrow \mathbb{R}^n$, whose derivative at $\varphi=0$  satisfying
  {\small
  \begin{eqnarray*}
    \left.\frac{\partial \Psi^{-1}(\varphi)}{\partial \varphi}\right|_{\varphi=0} &=& \left(\left.\frac{\partial \Psi(\bar{x})}{\partial \bar{x}}\right|_{\bar{x}=0} \right)^{-1}\\
    &=& \begin{bmatrix} \Lambda & -\Lambda P_2\bar{F}_2^{-1} \\ -\bar{F}_2^{-1}\bar{F}_1\Lambda   & \;\;\bar{F}_2^{-1}\!+\!\bar{F}_2^{-1}\bar{F}_1\Lambda P_2\bar{F}_2^{-1}\end{bmatrix},
  \end{eqnarray*}}
  where $\Lambda= \big(P_1-P_2\bar{F}_2^{-1}\bar{F}_1\big)^{-1}$.

  By (\ref{eq:Method:stable4NLS_Proof1}), this inverse mapping can be denoted as a function of two variables, $\Psi^{-1}(\varphi_\Romannum{1},\varphi_\Romannum{2})$. For any trajectory $x(t)$ of (\ref{NLS}) satisfying $x(t) \in \mathcal{M} $, we have $G(x)\equiv0$, i.e., $\varphi_\Romannum{2} \equiv 0$. Thus we only need to consider the motion of $\varphi_\Romannum{1}$. The dynamics of $\varphi_\Romannum{1}$ is
    \begin{equation}\label{eq:Method:stable4NLS_LinearizedProof1}
      \dot{\varphi_\Romannum{1}}=\left.Pf(x)\right|_{x=\Psi^{-1}(\varphi_\Romannum{1},0)+x_e}.
    \end{equation}
  The linearization of dynamics (\ref{eq:Method:stable4NLS_LinearizedProof1}) around $\varphi_\Romannum{1}=0$ gives
  \begin{eqnarray*}
    \dot{\varphi}_\Romannum{1} &=& P\left.\frac{\partial f}{\partial x}\right|_{x=x_e} \left.\frac{\partial \Psi^{-1}}{\partial \varphi_\Romannum{1}}\right|_{\varphi_\Romannum{1}=0} \varphi_\Romannum{1}\\
    &=& P\bar{A} \begin{bmatrix} I_{n-m}\\ -\bar{F}_2^{-1}\bar{F}_1\end{bmatrix} \left( P\begin{bmatrix} I_{n-m}\\ -\bar{F}_2^{-1}\bar{F}_1 \end{bmatrix} \right)^{-1}   \varphi_\Romannum{1}.
  \end{eqnarray*}
  Due to (\ref{eq:Method:stable4NLS_Con2}), system (\ref{eq:Method:stable4NLS_LinearizedProof1}) is exponentially stable, which implies the assertion.
\end{proof}

\subsection{Stability Analysis of Platonic Formations}

In this subsection, in order to investigate stability of equilibrium $\xi_s= \xi_s^{\{p,q\}}$ in subsystem (\ref{eq:SF:newCorDynf1}) restricted to manifold $\mathcal{M}_C$ in (\ref{eq:Method:constraintManifold}), a method is proposed to show that it is sufficient to examine stability of (\ref{eq:SF:newCorDynf1}) restricted to a simpler manifold $\mathcal{\overline{M}}_C$ defined by less constraints.

The aim of this method is twofold, one is obviously that less
constraints considered simplify the process of stability
analysis. What is more, by implicit function theorem, in order to substitute $m_0=D_\xi-2N_0$
redundant variables in $\xi$, we need $m_0$ nonsingular constraints
$g_1, \cdots, g_{m_0}$ having $\left.\frac{\partial G}{\partial
    \xi_s}\right|_{\xi_s=\xi_s^{\{p,q\}}}$ with full row rank, where
$G(\xi_s)=[g_{1}(\xi_s), \cdots, g_{m_0}(\xi_s)]^T$. With the help of the proposed method, the stability of Platonic formations can be achieved, even though less than $m_0$ nonsingular constraints can be found.

In the following, we restrict control gain function $h(\cdot)$ to a concrete form. For other control gain functions, the stability analysis can be done with a same procedure.
\begin{assumption}\label{As:PF:h}
  The gain function $h(\cdot)$ in (\ref{eq:PF:Control}) has the structure $h(\theta_{ij})=\exp(2\cos(\theta_{ij}))$.
\end{assumption}

We denote $\hat{A}=\left.\frac{\partial \hat{f}_s(\xi_s)}{\partial \xi_s}\right|_{\xi_s=\xi_s^*}$, where $\xi_s^*=\xi_s^{\{p,q\}}$ is the equilibrium of subsystem (\ref{eq:SF:newCorDynf1}) corresponding to formation $\mathcal{M}_{\{p,q\}}$. Then the next theorem shows that stability of the system restricted to a less constrained manifold $\mathcal{\overline{M}}_C$ can imply that to $\mathcal{M}_C$.
\begin{theorem}
  Let $C_1, C_2, \dots, C_m$ be a sequence of 4-element subsets of $\mathcal{V}$. Denote $G(\xi_s)\!=\![g_{C_1}\!(\xi_s), \!\cdots\!, g_{C_m}\!(\xi_s)]^T$ with $g_{C_i}(\cdot)$ following definition (\ref{eq:Method:Constraint}). In closed-loop system (\ref{eq:PF:ClosedSystem}), the desired formation $\mathcal{M}_{\{p,q\}}$ is exponentially stable, if
  \begin{enumerate}[label=(\alph*)]
    \item $\left.\frac{\partial G}{\partial \xi_s}\right|_{\xi_s=\xi_s^*}$ has full row rank,
    \item $T_1 \hat{A} T_1^T$ is Hurwitz,
  \end{enumerate}
  where $T_1=[v_1,\cdots,v_r]^T$ whose rows constitute an orthonormal basis of the maximal $\hat{A}$-invariant subspace containing in $\mathbf{V}=\big\{\xi_s: \left.\frac{\partial G}{\partial \xi_s}\right|_{x=\xi_s^*} \xi_s =0 \big\}$.
\end{theorem}

\begin{proof}
  According to (a) and (b), Proposition~\ref{Thm:iff_SRS} gives the linearized system $\dot{\overline{\xi}}_s=\hat{A}\overline{\xi}_s$ restricted to $\mathbf{V}$ is stable where $\overline{\xi}_s=\xi_s-\xi_s^*$. Then, by Theorem~\ref{Thm:Method:NonlinearRestrict}, nonlinear system (\ref{eq:SF:newCorDynf1}) restricted to manifold $\mathcal{\overline{M}}_C =\{\xi_s: g_{C_i}(\xi_s)=0, i=1, \cdots, m\}$ is exponentially stable at $\xi_s^*$.

  Let $\mathcal{M}_C$ follows the definition in (\ref{eq:Method:constraintManifold}) and function $V_i(t)=\Gamma_i^T\Gamma_i$. The derivative of $V_i(t)$ along the trajectory of system (\ref{eq:PF:ClosedSystem}) satisfies
  $\dot{V}_i(t)=2\Gamma_i^T\widehat{\Gamma}_i^2 \sum_{j\in \mathcal{N}_i} h(\theta_{ij})\Gamma_j=0,$
  which implies $\Gamma_i(t) \in \mathcal{S}^2, \forall t>0$ if $\Gamma_i(0) \in \mathcal{S}^2$. By Lemma~\ref{Thm:Method:IdentityFor4P}, this leads to $\mathcal{M}_C$ being an invariant manifold in system (\ref{eq:SF:newCorDynf1}).

  As (\ref{eq:SF:newCorDynf1}) restricted to manifold $\mathcal{\overline{M}}_C$ is exponentially stable at $\xi_s^*$, there is a neighbourhood of $\xi_s^*$, denoted by $B(\xi_s^*)$. For any trajectory $\xi_s(t)$ of (\ref{eq:SF:newCorDynf1}) satisfying $\xi_s(0) \in B(\xi_s^*) \cap \mathcal{M}_C$, due to the invariance of $\mathcal{M}_C$ and the fact that $\mathcal{M}_C \subset \mathcal{\overline{M}}_C$, we have $\xi_s(t)\in \mathcal{\overline{M}}_C$, $\forall t  \geq 0$. This implies that there are positive constants $\alpha, \beta$ such that $\|\xi_s(t)-\xi_s^*\| < \alpha e^{-\beta t}$, $\forall t  \geq 0$.

  By Remark~\ref{Rmk:SF:maniToPoint1}, we have the desired formation $\mathcal{M}_{\{p,q\}}$ is exponentially stable.
\end{proof}


Following the above theorem, Algorithm~1 is given to verify exponential stability of formation $\mathcal{M}_{\{p,q\}}$ for closed-loop system (\ref{eq:PF:ClosedSystem}). In  Algorithm~1, we denote $\sigma(\hat{A})$ as the set of all eigenvalues of matrix $\hat{A}$, and $\mathbb{C}^-$ as the left-half complex plane.

\begin{figure}[t]
\resizebox{0.95\columnwidth}{!}{%
{\small
\begin{tabular}{p{\columnwidth}}
\hline
{\bf Algorithm 1: Stability Test for Platonic Solids}\\
\hline
\textbf{Step 1}. Set $m=0$, $\Upsilon_0=\emptyset$, and $T_0=I$.\\
\textbf{Step 2}. If $\sigma(T_m\hat{A}T_m^T) \subset \mathbb{C}^-$, go to Step 5.\\
\quad \quad \quad If $Card[\sigma(T_m\hat{A}T_m^T) \cap \mathbb{C}^-]<2N_0-3$, go to Step 6.\\
\textbf{Step 3}. Set $m=m+1$, and $\Upsilon_m=\Upsilon_{m-1} \cup \{C_m\}$, \\
\quad \quad \quad  where $C_m$ is a 4-element subset of $\mathcal{V}$, $\big\{g_S(\xi_s)=0:$ \\
\quad \quad \quad   $S\in\Upsilon_m \big\}$ are nonsingular constraints at $\xi_s^*$. \\
\quad \quad \quad  Denote $G_m(\xi_s)=\big[g_{S}(\xi_s)\big]_{S\in\Upsilon_m}$.\\
\textbf{Step 4}. Compute matrix $T_m=[v_1^m,\cdots,v_{r_m}^m]^T$ whose rows \\
\quad \quad \quad constitute an orthonormal basis of the maximal \\
\quad \quad \quad $\hat{A}$-invariant subspace containing in $\mathbf{V_m}=\big\{\xi_s:$ \\
\quad \quad \quad $ \left.\frac{\partial G_m}{\partial \xi_s}\right|_{\xi_s=\xi_s^*} \xi_s =0 \big\}$. Go to Step 2.\\
\textbf{Step 5}. The formation $\mathcal{M}_{\{p,q\}}$ is exponentially stable. \\
\quad \quad \quad Let $m_{max}=m$. End the algorithm.\\
\textbf{Step 6}. $\mathcal{M}_{\{p,q\}}$ is not exponentially stable. \\
\quad \quad \quad End the algorithm.\\
\hline
\end{tabular}
}}
\end{figure}

Although we can obtain all possible graphs satisfying the
symmetries assumption according to Section~\ref{sec:graph}, in order to reach the conclusion in a compact way,
we restrict the inter-agent graphs to some specific ones by the following assumption. We note that the systems under other possible graphs can also be investigated in the same manner.

\begin{assumption}\label{ass:GD:graphChoice}
  The inter-agent topology $\mathcal{G}_{\{p,q\}}$  is a complete graph, if $p=3$. Graph $\mathcal{G}_{\{4,3\}}$ and $\mathcal{G}_{\{5,3\}}$ are listed in Fig.~\ref{Fig:GD:Graph_820}(a) and Fig.~\ref{Fig:GD:Graph_820}(b) respectively.
\end{assumption}

\begin{figure}[ht] 
    \centering
    \subfigure[Inter-agent graph $\mathcal{G}_{\{4,3\}}$]{%
        \includegraphics[width=0.13\textwidth]{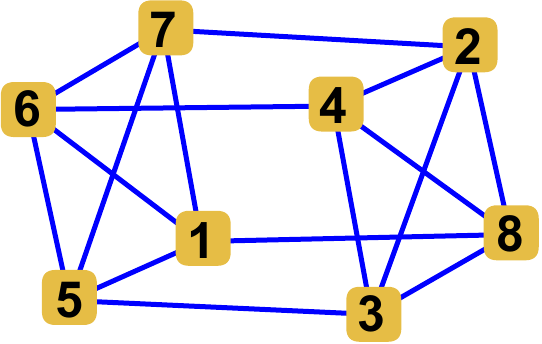}
        \label{Fig:GD:Graph_N8}}\!
    \subfigure[Inter-agent graph $\mathcal{G}_{\{5,3\}}$]{%
        \includegraphics[width=0.13\textwidth]{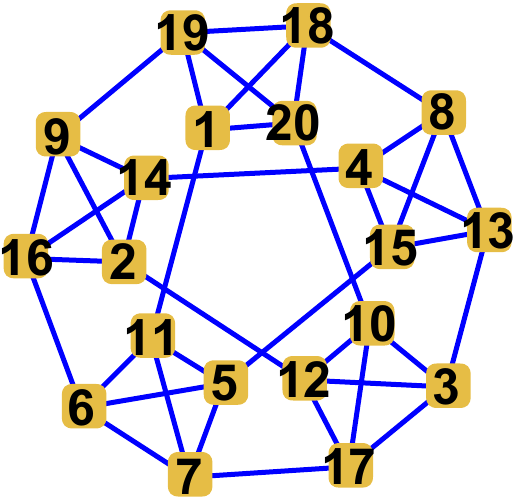}
        \label{Fig:GD:Graph_N20}}\!
    \subfigure[A $\{4,3\}$ compound]{%
        \includegraphics[width=0.09\textwidth]{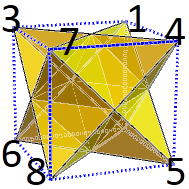}
        \label{Fig:Int:2Tetrahedron}}\!
    \subfigure[A $\{5,3\}$ compound]{%
        \includegraphics[width=0.09\textwidth]{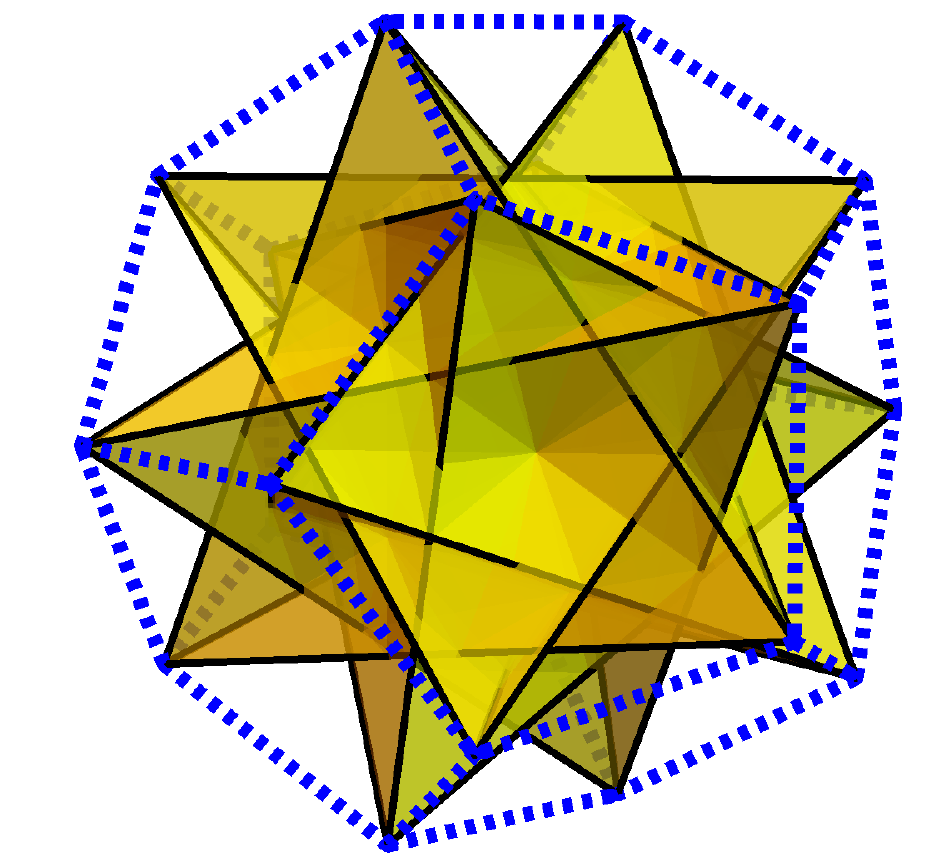}
        \label{Fig:Int:5Tetrahedron}}\!
    \caption{Graphs of $\{4,3\}$ and $\{5,3\}$: (c) depicts a cube composed by tetrahedra $\{1,5,6,7\}$ and $\{3,4,8,2\}$, which corresponds to the graph in (a). And (b),(d) have the similar correspondence. }
    \label{Fig:GD:Graph_820}
\end{figure}

Then stability of the five Platonic solids can be achieved by Algorithm~1, we state this result in the next Proposition.
\begin{proposition}
  Under Assumption~\ref{As:PF:h} and Assumption~\ref{ass:GD:graphChoice}, in closed-loop system (\ref{eq:PF:ClosedSystem}), the regular polyhedra formations $\mathcal{M}_{\{p,q\}}$ entail exponential stability, for all integer $p$, $q$ satisfying inequality (\ref{eq:pre:relation_p_q}).
\end{proposition}

This proposition is obtained by applying Algorithm~1. For the formations $\{3,3\}$, $\{4,3\}$, and $\{3,5\}$, we obtain $m_{max}=0$ and $\Upsilon_0=\emptyset$ when the algorithm achieves their stabilities.
As investigating the formation $\{3,4\}$, the algorithm provides its stability  with $m_{max}=3$ and the nonsingular constraints $\Upsilon_3=\big\{g_{\{1,2,3,i\}} \big\}_{i\in S_3}$, where $S_3=\{4,5,6\}$. In the scenario of dodecahedron  $\{5,3\}$, the algorithm ends with $m_{max}=12$ and the nonsingular constraints $\Upsilon_{12}=\big\{g_{\{1,2,3,i\}} \big\}_{i\in S_{12}}$, where $S_{12}=\{4,5,6,7,8,9,10,11,12,13,14,15\}$. \rev{We can see that by the virtue of the method proposed, the number of regular constraints needed in stability analysis $m_{max}$ is far less than the number of redundant variables $m_0$ in $\xi$, for instant originally $m_0=45$ for solid $\{3,5\}$ and $m_0=153$ for solid$\{5,3\}$.}

To conclude this section, we provide some clues to the choice of graphs in Assumption~\ref{ass:GD:graphChoice}.
A complete graph can solve Problem~1 for $\{p,q\}$ with $p=3$, but not for solids $\{4,3\}$ and $\{5,3\}$. This is because the faces of $\{4,3\}$ and $\{5,3\}$ are regular quadrilaterals and regular pentagons respectively, which are not structurally rigid and prone to be bent or flexed under lateral interactions \cite{crapo1979structural}. 
To construct these two non-rigid polyhedra, we use the regular tetrahedra as \emph{building blocks}.

In light of the concept of polyhedral compounds, as shown in Fig.~\ref{Fig:GD:Graph_820}(c-d), the vertices of a two-tetrahedron compound can compose a cube, and those of a five-tetrahedron compound can build a dodecahedron. Thus we employ the graphs constituted by associating two and five $\mathcal{P}_{\{3,3\}}$ as inter-agent topologies for $\{4,3\}$ and $\{5,3\}$, which are shown in Fig.~\ref{Fig:GD:Graph_820}(a-b).
We note that this machinery of constructing formations by the simplex
blocks is potentially applicable to more formation problems.



\section{Simulation}\label{sec:simu}
In this section, we present some numerical simulation results to illustrate the convergence of desired formations governed by the proposed control (\ref{eq:PF:Control}).
Under Assumption~\ref{As:PF:h} and Assumption~\ref{ass:GD:graphChoice}, the trajectories of closed-loop system (\ref{eq:PF:ClosedSystem}) with random initial conditions are simulated. The resulting trajectories for the five Platonic solids are shown in Fig.~\ref{Fig:Sim:5solids} (a)-(e) respectively, in which the initial reduced attitudes are marked by pentagrams and the final states are marked by circles.


\begin{figure}[ht]
    \subfigure[Trajectory of $\{3,3\}$]{%
        \includegraphics[width=0.21\textwidth]{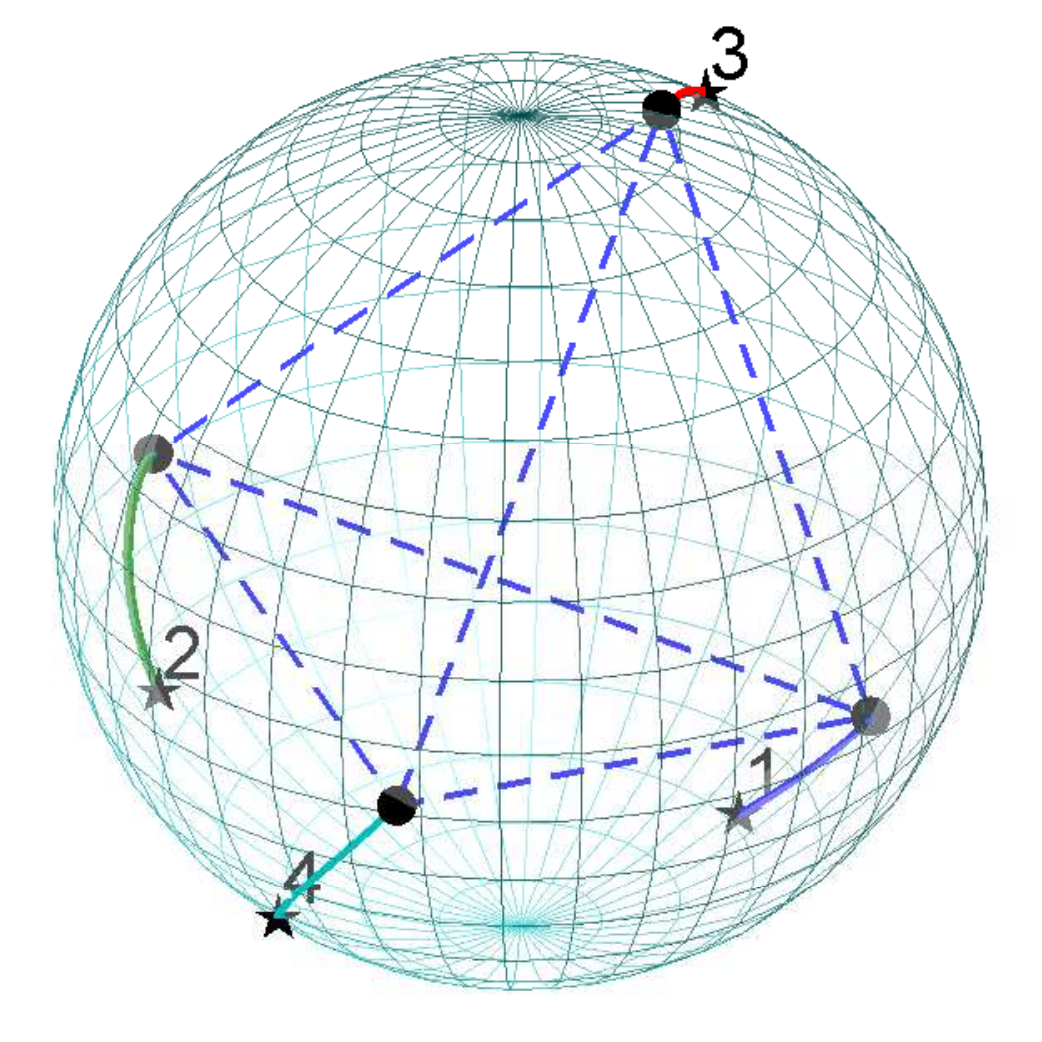}
        \label{Fig:Sim:Tetrahedron}} \,
    \subfigure[Trajectory of $\{3,4\}$,]{%
        \includegraphics[width=0.21\textwidth]{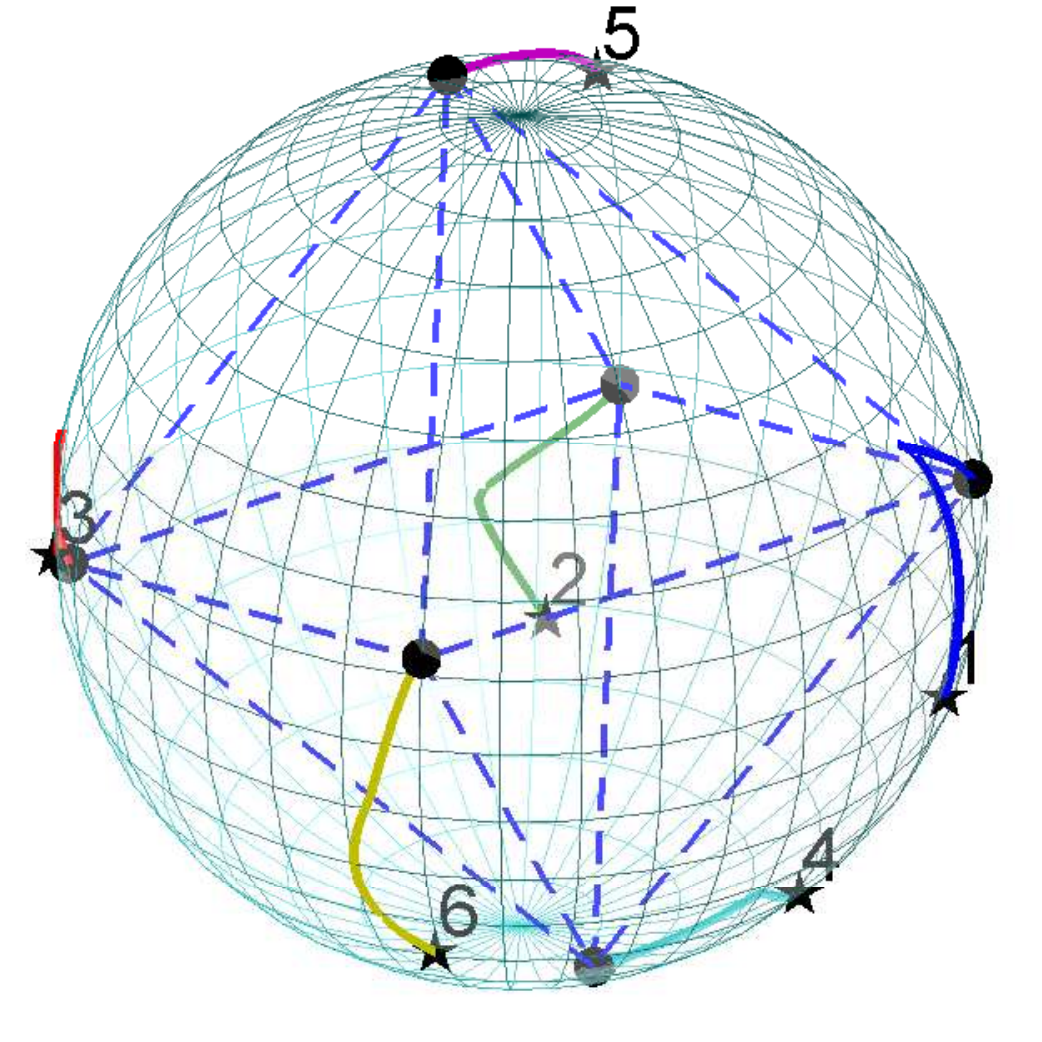}
        \label{Fig:Sim:Octahedron}}
    \subfigure[Trajectory of $\{3,5\}$,]{%
        \includegraphics[width=0.21\textwidth]{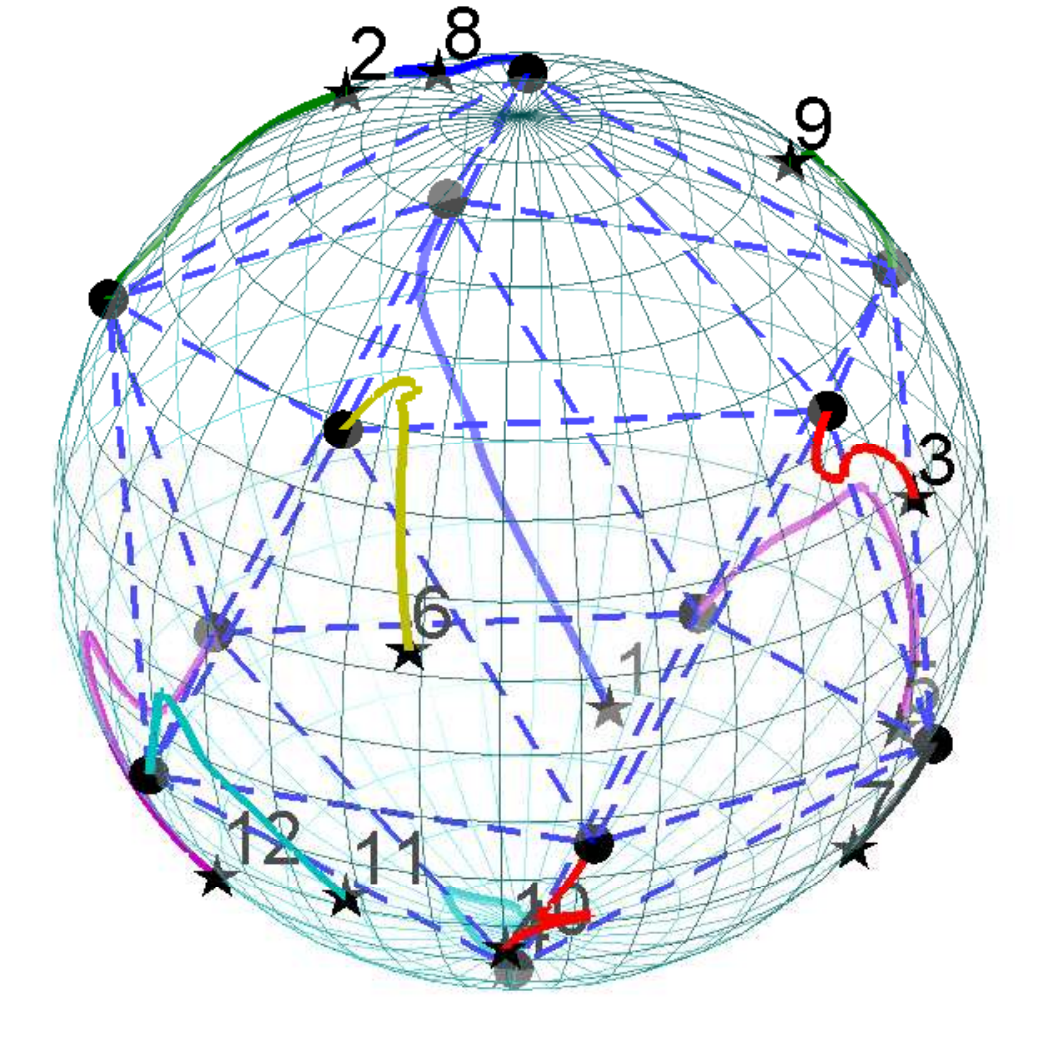}
        \label{Fig:Sim:Icosahedron}}\,
    \subfigure[Trajectory of $\{4,3\}$]{%
        \includegraphics[width=0.21\textwidth]{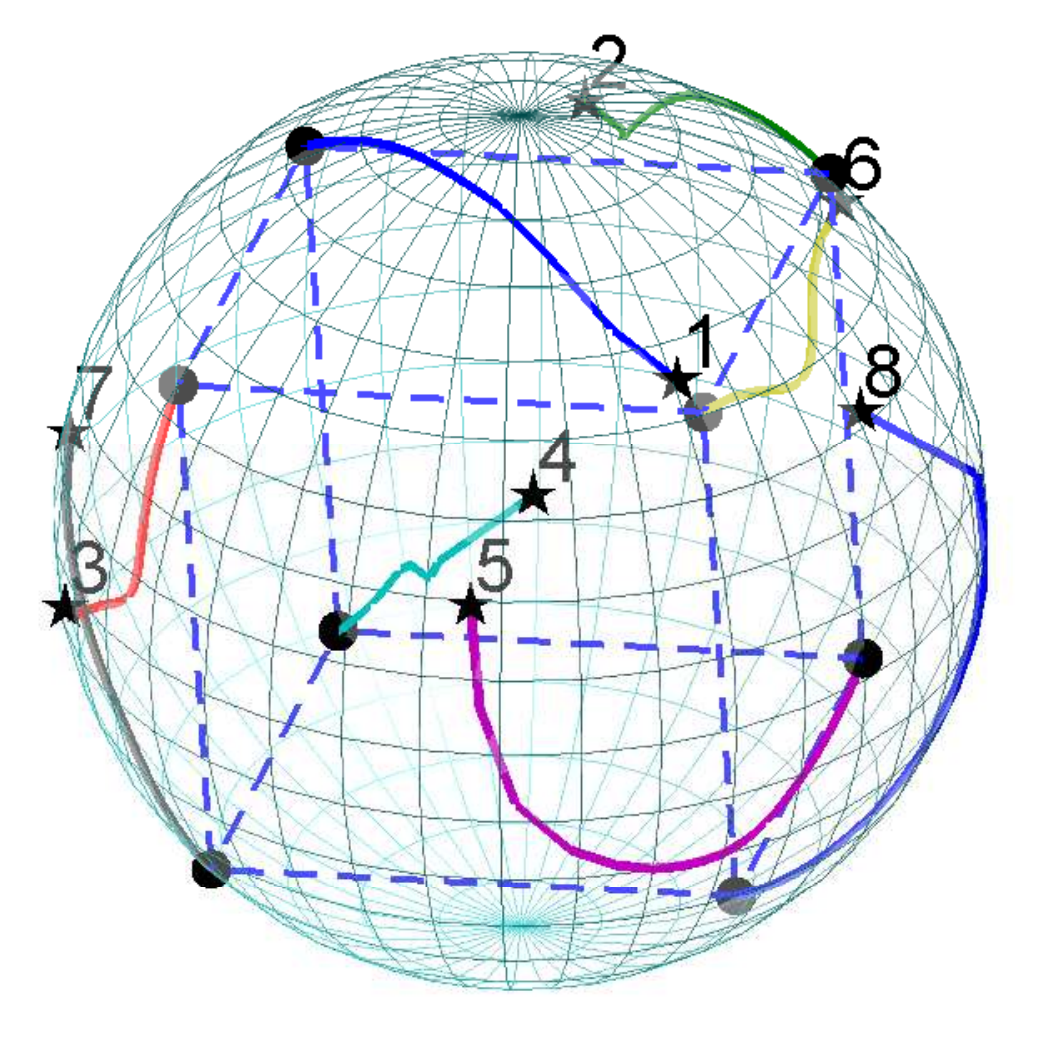}
        \label{Fig:Sim:Cube}}
    \subfigure[Trajectory of $\{5,3\}$]{%
        \includegraphics[width=0.21\textwidth]{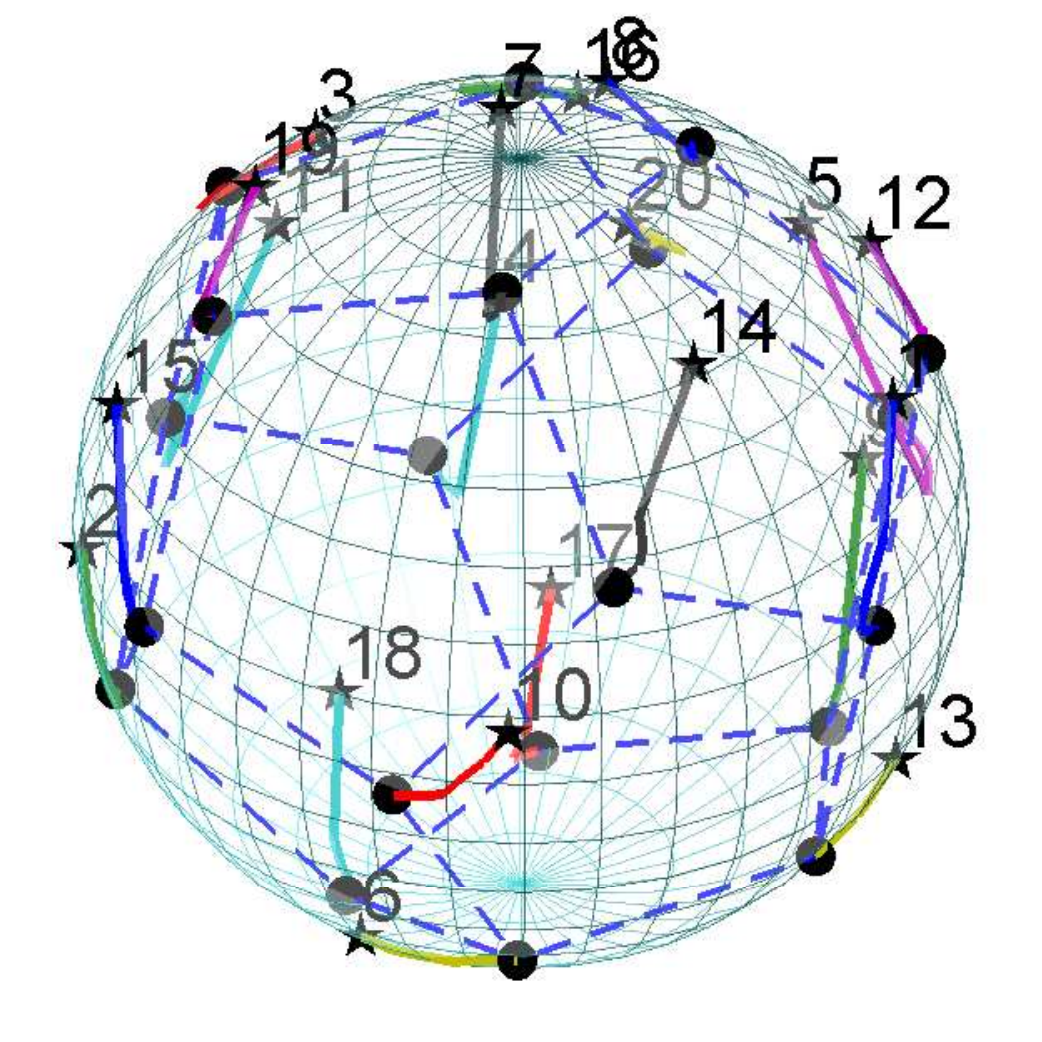}
        \label{Fig:Sim:Dodecahedron}}\centering
    \caption{Simulation results of reduced attitude systems for five Platonic solids}
    \label{Fig:Sim:5solids}
\end{figure}

\section{Conclusion and Future Work}\label{sec:conclusion}
This work studies formation control of reduced attitudes for regular
polyhedra patterns. The proposed method does not need to contain
formation error in the control law
to reduce the "distance" from the current formation to the desired formation, and shows that it is indeed possible to obtain formation by the geometry of space and the inter-agent topology with a
relatively simple control law. In the future work, how to achieve almost-global stability of Platonic solids formations will be investigated.


\bibliographystyle{plain}      
\bibliography{ref}           

\begin{thebibliography}{10}

\bibitem{Bhat00scl}
S.~Bhat and D.~Bernstein.
\newblock A topological obstruction to continuous global stabilization of
  rotational motion and the unwinding phenomenon.
\newblock {\em Systems \& Control Letters}, 39(1):63--70, 2000.

\bibitem{Bower1964}
J.~Bower and G.~Podraza.
\newblock {Digital implementation of time-optimal attitude control}.
\newblock {\em IEEE Transactions on Automatic Control}, 9(4):590--591, 1964.

\bibitem{bullo1995control}
F.~Bullo, R.~Murray, and A.~Sarti.
\newblock Control on the sphere and reduced attitude stabilization.
\newblock 1995.

\bibitem{Byrnes1991}
C.~Byrnes and A.~Isidori.
\newblock {On the attitude stabilization of rigid spacecraft}.
\newblock {\em Automatica}, 27(1):87--95, jan 1991.

\bibitem{chaturvedi2011rigid}
N.~Chaturvedi, A.~Sanyal, and N~H. McClamroch.
\newblock Rigid-body attitude control.
\newblock {\em Control Systems, IEEE}, 31(3):30--51, 2011.

\bibitem{cortes2002coverage}
J.~Cortes, S.~Martinez, T.~Karatas, and F.~Bullo.
\newblock Coverage control for mobile sensing networks.
\newblock In {\em Proceedings. ICRA'02.}, volume~2, pages 1327--1332. IEEE,
  2002.

\bibitem{coxeter1973regular}
H.~S.~M. Coxeter.
\newblock {\em Regular polytopes}.
\newblock Courier Corporation, 1973.

\bibitem{crapo1979structural}
H.~Crapo.
\newblock Structural rigidity.
\newblock {\em Structural topology, 1979, n{\'u}m. 1}, 1979.

\bibitem{credland1997cluster}
J.~Credland, G.~Mecke, and J.~Ellwood.
\newblock The cluster mission: Esa's spacefleet to the magnetosphere.
\newblock {\em Space Science Reviews}, 79(1):33--64, 1997.

\bibitem{peter1999polyhedra}
P.~R. Cromwell.
\newblock {\em Polyhedra}.
\newblock Cambridge University Press, 1999.

\bibitem{curtis1999magnetospheric}
S.~Curtis.
\newblock The magnetospheric multiscale mission... resolving fundamental
  processes in space plasmas.
\newblock In {\em Report of the NASA Science and Technology Definition Team for
  the Magnetospheric Multiscale (MMS) Mission}, volume~1, 1999.

\bibitem{dorrie2013100Problem}
H.~D{\"o}rrie.
\newblock {\em 100 great problems of elementary mathematics}.
\newblock Courier Corporation, 2013.

\bibitem{dunlop1998multi}
M.~Dunlop and T.~Woodward.
\newblock Multi-spacecraft discontinuity analysis: Orientation and motion.
\newblock {\em Analysis Methods for Multi-Spacecraft Data}, pages 271--306,
  1998.

\bibitem{goldberger2005reconstructing}
J.~Goldberger.
\newblock Reconstructing camera projection matrices from multiple pairwise
  overlapping views.
\newblock {\em Computer Vision and Image Understanding}, 97(3):283--296, 2005.

\bibitem{gracias2000forming}
D.~Gracias, J.~Tien, T.~Breen, C.~Hsu, and G.~Whitesides.
\newblock Forming electrical networks in three dimensions by self-assembly.
\newblock {\em Science}, 289(5482):1170--1172, 2000.

\bibitem{hughes2008formation}
S.~Hughes.
\newblock Formation design and sensitivity analysis for the magnetospheric
  multiscale mission (mms).
\newblock In {\em AIAA/AAS Astrodynamics Specialist Conference, Honolulu, HI},
  2008.

\bibitem{kim2010spherical}
J.~Kim, M.~Hwangbo, and T.~Kanade.
\newblock Spherical approximation for multiple cameras in motion estimation:
  Its applicability and advantages.
\newblock {\em Computer Vision and Image Understanding}, 114(10):1068--1083,
  2010.

\bibitem{Kowalik1970}
H.~Kowalik.
\newblock {A spin and attitude control system for the Isis-I and Isis-B
  satellites}.
\newblock {\em Automatica}, 6(5):673--682, sep 1970.

\bibitem{lawton2002synchronized}
J.~R. Lawton and R.~W. Beard.
\newblock Synchronized multiple spacecraft rotations.
\newblock {\em Automatica}, 38(8):1359--1364, 2002.

\bibitem{markdahl2017almost}
J.~Markdahl, J.~Thunberg, and J.~Gon{\c{c}}alves.
\newblock Almost global consensus on the $n$-sphere.
\newblock {\em IEEE Transactions on Automatic Control}, 2017.

\bibitem{min2011distributed}
H.~Min, S.~Wang, F.~Sun, Z.~Gao, and Y.~Wang.
\newblock Distributed six degree-of-freedom spacecraft formation control with
  possible switching topology.
\newblock {\em IET control theory \& applications}, 5(9):1120--1130, 2011.

\bibitem{ZXLi}
R.~Murray, Z.~Li, S.~Sastry, and S.~Sastry.
\newblock {\em A mathematical introduction to robotic manipulation}.
\newblock CRC press, 1994.

\bibitem{paley2009stabilization}
D.~Paley.
\newblock Stabilization of collective motion on a sphere.
\newblock {\em Automatica}, 45(1):212--216, 2009.

\bibitem{pless2003using}
R.~Pless.
\newblock Using many cameras as one.
\newblock In {\em Computer Vision and Pattern Recognition, 2003. IEEE Computer
  Society Conference on}, volume~2, pages II--587. IEEE, 2003.

\bibitem{porfiri2007tracking}
M.~Porfiri, D~G. Roberson, and D.~J Stilwell.
\newblock Tracking and formation control of multiple autonomous agents: A
  two-level consensus approach.
\newblock {\em Automatica}, 43(8):1318--1328, 2007.

\bibitem{Ren03ACC}
W.~Ren and R.~Beard.
\newblock Decentralized scheme for spacecraft formation flying via the virtual
  structure approach.
\newblock {\em Journal of Guidance, Control, and Dynamics}, 27(1):73--82, 2004.

\bibitem{roscoe2013satellite}
C.~Roscoeand, S.~Vadali, K.~Alfriendand, and P.~Desai.
\newblock Satellite formation design in orbits of high eccentricity with
  performance constraints specified over a region of interest: Mms phase ii.
\newblock {\em Acta Astronautica}, 82(1):16--24, 2013.

\bibitem{Sciavicco_Book}
L.~Sciavicco and B.~Siciliano.
\newblock {\em Modelling and control of robot manipulators}.
\newblock Springer Science \& Business Media, 2012.

\bibitem{WJ15ac}
W.~Song, J.~Markdahl, S.~Zhang, X.~Hu, and Y.~Hong.
\newblock Intrinsic reduced attitude formation with ring inter-agent graph.
\newblock {\em Automatica}, 85:193--201, 2017.

\bibitem{Todhunter_Identity}
I.~Todhunter and J.~Leathem.
\newblock {\em Spherical trigonometry: for the use of colleges and schools}.
\newblock Macmillan and Company, 1914.

\bibitem{Tsiotras1994}
P.~Tsiotras and J.~M. Longuski.
\newblock {Spin-axis stabilization of symmetric spacecraft with two control
  torques}.
\newblock {\em Systems {\&} Control Letters}, 23(6):395--402, 1994.

\bibitem{Wang13PRL}
X.~Wang.
\newblock Intelligent multi-camera video surveillance: A review.
\newblock {\em Pattern recognition letters}, 34(1):3--19, 2013.

\bibitem{wonham1974linear}
W~M. Wonham.
\newblock {\em Linear multivariable control}.
\newblock Springer, 1985.

\bibitem{zhang2017tetrahedron}
S.~Zhang, W.~Song, F.~He, Y.~Hong, and X.~Hu.
\newblock Intrinsic tetrahedron formation of reduced attitude.
\newblock {\em Automatica}, 87:375--382, 2018.

\end{thebibliography}



\appendix
\section{Proofs}

\subsection{Proof of \emph{Proposition} \ref{thm:Int:twoMEquivalent}} \label{sec:AppendixPfM}
It is obvious that $\mathcal{M}_{\{p,q\}}\subset \mathcal{M}'_{\{p,q\}}$. We need to prove another side.
For any vertex $\ell \in \mathcal{V}$,  the corresponding symmetry is $(R_\ell, \sigma_\ell) \in \mathcal{H}_{\{p,q\}}$. Denote the cycle containing vertex $k \in \mathcal{V}$ in the cycle decomposition of $\sigma_\ell$ by $\mathcal{C}_\ell(k)$.  Then the definition \eqref{eq:Int:equilibriaSet2} implies that for $\mathbf{\Gamma}^* \in \mathcal{M}'_{\{p,q\}}$,
\begin{equation}\label{eq:Int:pf1}
 \Gamma_{\sigma_\ell(i)}^*=R_\ell(\mathbf{\Gamma}^*)\Gamma_i^*.
\end{equation}

Next we show that for $\mathbf{\Gamma}^* \in \mathcal{M}'_{\{p,q\}}$, it holds that ${\Gamma}^*_i \neq  \Gamma^*_j$, $\forall i \neq j$. Since $\mathbf{\Gamma}^* \in \mathcal{M}'_{\{p,q\}}$ there exist $m$ and $n$ such that $\widehat{\Gamma}_m \Gamma_n\neq 0 $. We consider the permutation $\sigma_m$. Its cycle $\mathcal{C}_m(m)$ has to be a singleton cycle which is $(m)$, since $R_m(\mathbf{\Gamma}^*)=\exp(\frac{2\pi}{q}\widehat{\Gamma}_m^*)$ is a rotation about reduced attitude $m$. Moreover $\mathcal{C}_m(n)$ must contain elements other than $n$, otherwise ${\Gamma}^*_m = \pm\Gamma^*_n $. According to \eqref{eq:Int:pf1}, all reduced attitudes $\Gamma^*_j$ for $j \in \mathcal{C}_m(n)$ do not overlap with each other and $\Gamma^*_j \neq \pm\Gamma^*_m$. Then by repeating the forgoing procedure for the vertices already shown to be mutually unequal, we can get all reduced attitudes in $\mathbf{\Gamma}^*$ do not overlap with each other.

Furthermore, 
for vertex $i \in \mathcal{V}$ the non-sigleton cycle $\widehat{\mathcal{C}}_i$ in permutation $\sigma_i$ is defined as the segment between $i$ and any $j\in \widehat{\mathcal{C}}_i$ is an edge of $\{p,q\}$. Due to \eqref{eq:Int:pf1}, all these segments $(i, j)$ are of the same length $\forall j\in \widehat{\mathcal{C}}_i$. This process can be extended to all the edges. Hence we obtain that $\mathbf{\Gamma}^*$ is a polyhedron with identical length of all edges, i.e., $\mathbf{\Gamma}^* \in \mathcal{M}_{\{p,q\}}$.
\hspace*{\fill}$\square$

\subsection{Proof of \emph{Lemma} \ref{thm:controlability}} \label{sec:AppendixPfControllability}
We consider a system
  \begin{align}\label{eq:Method:pfObservable1}
        &\dot{z}_2(t)= T_2AT_2^T z_2(t), \;\; z_2 \in \mathbb{R}^{(n-m-r)} \nonumber \\
        &y=T_3AT_2^T z_2(t). \nonumber
  \end{align}
Suppose $(T_3AT_2^T, T_2AT_2^T)$ is not observable, then there is a $z_2^*(0)\neq0$ such that the trajectory $z_2^*(t)=\exp^{(T_2AT_2^T)t} z_2^*(0)$ fulfills $ T_3AT_2^T z_2^*(t) = 0$ for any $t\geq 0$.

Then for a trajectory $z^*(t)=[z_1^{T}(t),z_2^{*T}(t),0^T]^T $, we have $x^*(t)=T^Tz^*(t) \in \mathbf{V}$ for $t \geq 0$. Moreover, if $z_1(t)$ satisfies $\dot{z}_1(t)=T_1AT_1^T z_1(t)+ T_1AT_2^T z_2(t) $, it can be verified that $\dot{z}^*(t)=[\dot{z}_1^{T}, T_2AT_2^T z_2^{*T},0^T]^T=A_z z^*(t) $, namely $z^*(t)$ is a solution of system (\ref{3PartTrans1}). However $x^*(t)=T^Tz^*(t) \notin \mathbf{V}^*$, which is a contradiction.
\hspace*{\fill}$\square$

\section{All graphs fulfilling Assumption~\ref{ass:GD:graphSymmetry}} \label{sec:AppendixGraph}
For tetrahedron $\{3,3\}$, the only possible graph satisfying Assumption~\ref{ass:GD:graphSymmetry} is the complete graph $\mathcal{K}_{4}$.
For Octahedron $\{3,4\}$, there exist two possible graphs, which are $\mathcal{K}_{6}$ and Platonic graph $\mathcal{P}_{\{3,4\}}$. In the case of cube $\{4,3\}$, beyond $\mathcal{K}_{8}$, $\mathcal{P}_{\{4,3\}}$ and the graph in Fig.~\ref{Fig:GD:Graph_N8}, there exist two other possible graphs shown in Fig.~B.\ref{Fig:Apen:Graph8_1} and Fig.~B.\ref{Fig:Apen:Graph8_2}.
When $N_0=12$ for icosahedron $\{3,5\}$, two graphs fulfilling Assumption~\ref{ass:GD:graphSymmetry} are listed in Fig.~B.\ref{Fig:Apen:Graph12_1} and Fig.~B.\ref{Fig:Apen:Graph12_2}
  other than two trivial graphs $\mathcal{K}_{12}$, $\mathcal{P}_{\{3,5\}}$. In the case for dodecahedron $\{5,3\}$, similarly $\mathcal{K}_{20}$, $\mathcal{P}_{\{5,3\}}$ and the graph in Fig.~\ref{Fig:GD:Graph_N20} are possible graphs. In addition to these three graphs, other $30$ connected graphs also fulfill the symmetries in Assumption~\ref{ass:GD:graphSymmetry}. Due to limit of space, here we sacrifice their detailed list.
\begin{figure}[ht]
    \centering
    \subfigure[$N_0=8$]{%
        \includegraphics[width=0.11\textwidth]{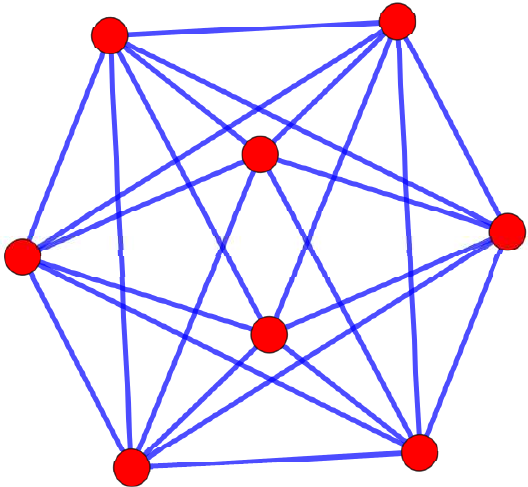}
        \label{Fig:Apen:Graph8_1}}\!
    \subfigure[$N_0=8$]{%
        \includegraphics[width=0.11\textwidth]{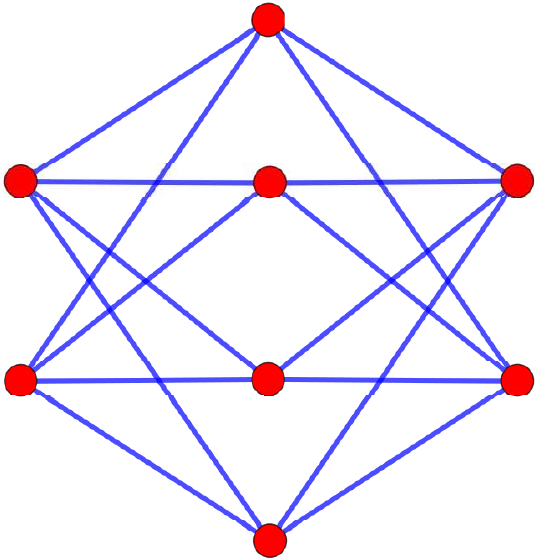}
        \label{Fig:Apen:Graph8_2}}\!
    \subfigure[$N_0=12$]{%
        \includegraphics[width=0.11\textwidth]{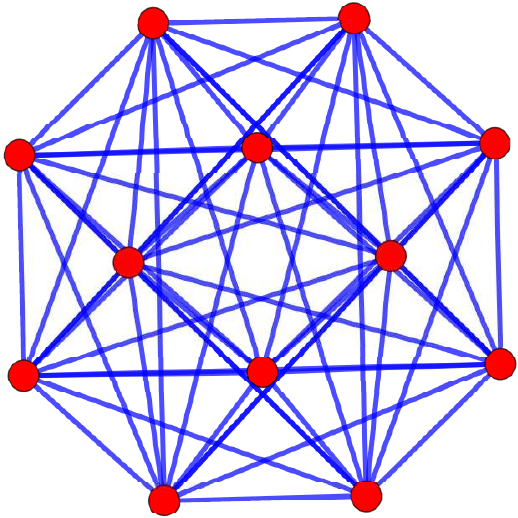}
        \label{Fig:Apen:Graph12_1}}\!
    \subfigure[$N_0=12$]{%
        \includegraphics[width=0.11\textwidth]{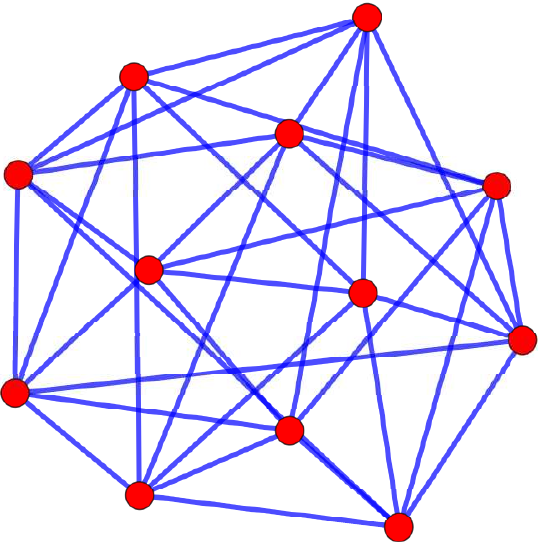}
        \label{Fig:Apen:Graph12_2}}

    \caption{List of Possible Graphs }
    \label{Fig:Int:PossibleGraphs}
\end{figure}

\section{Platonic Solids}
\begin{figure}[H]
    \centering
    \subfigure[$\{3,3\}$,\;Tetrahedron]{%
        \includegraphics[width=0.14\textwidth]{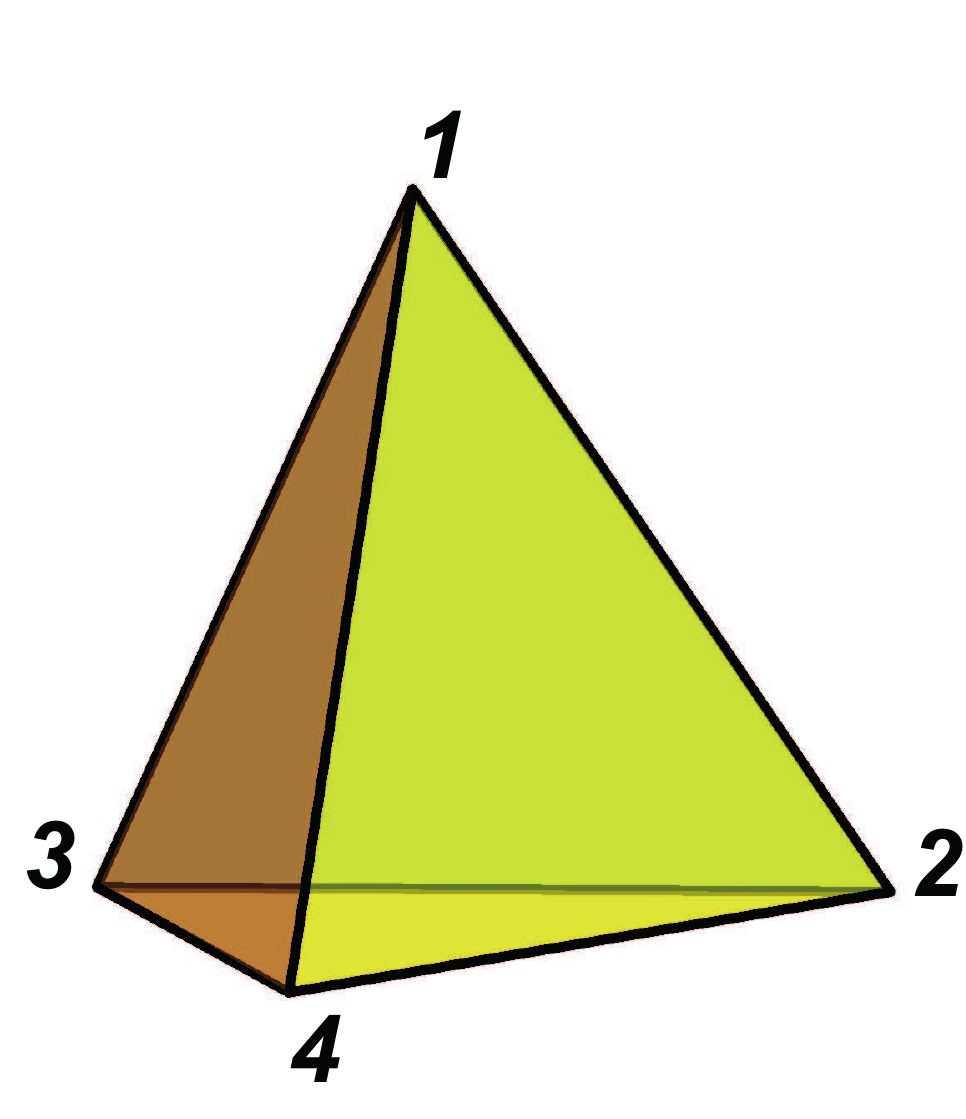}
        \label{Fig:Int:Tetrahedron}}\,
    \subfigure[$\{3,4\}$,\;Octahedron]{%
        \includegraphics[width=0.14\textwidth]{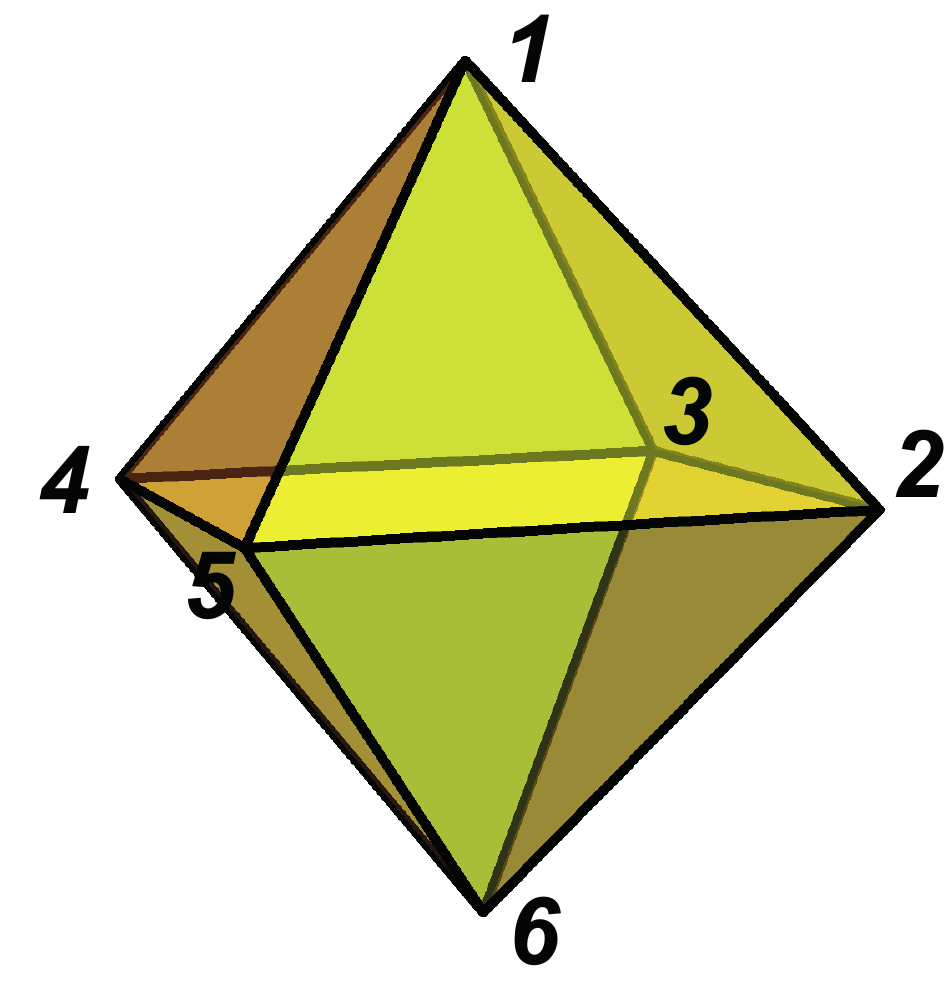}
        \label{Fig:Int:Octahedron}} \,
    \subfigure[$\{3,5\}$,\;Icosahedron]{%
        \includegraphics[width=0.14\textwidth]{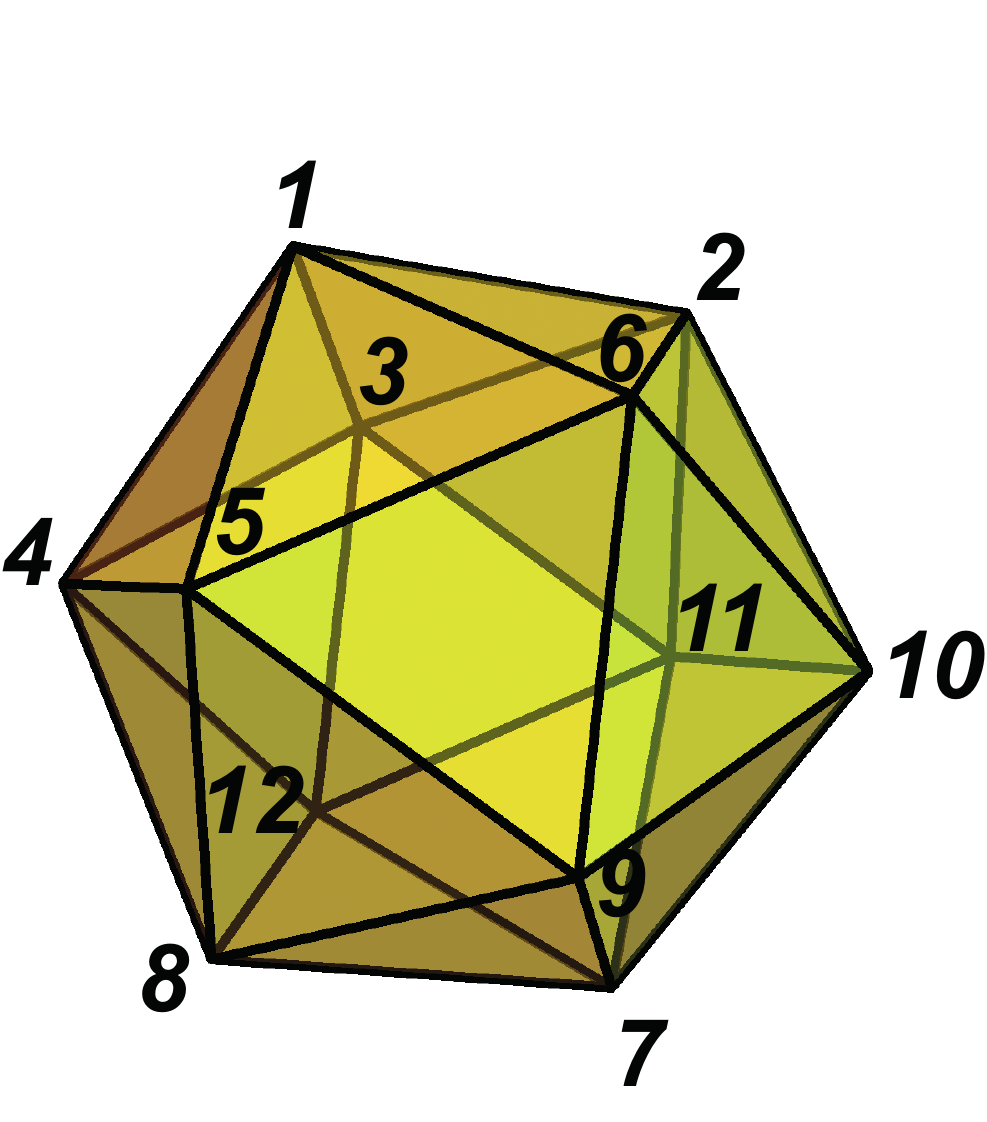}
        \label{Fig:Int:Icosahedron}}
    \subfigure[$\{4,3\}$,\;Cube]{%
        \includegraphics[width=0.14\textwidth]{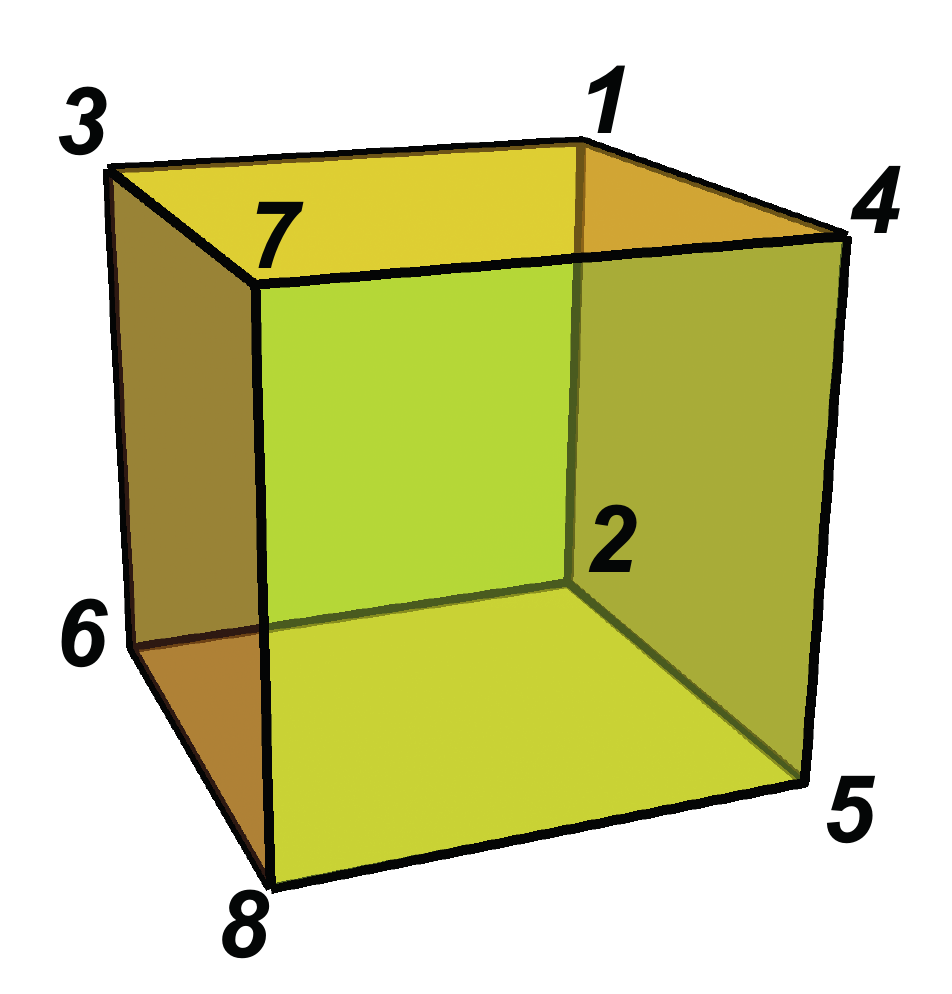}
        \label{Fig:Int:Cube}} \;
    \subfigure[$\{5,3\}$,\;Dodecahedron]{%
        \includegraphics[width=0.14\textwidth]{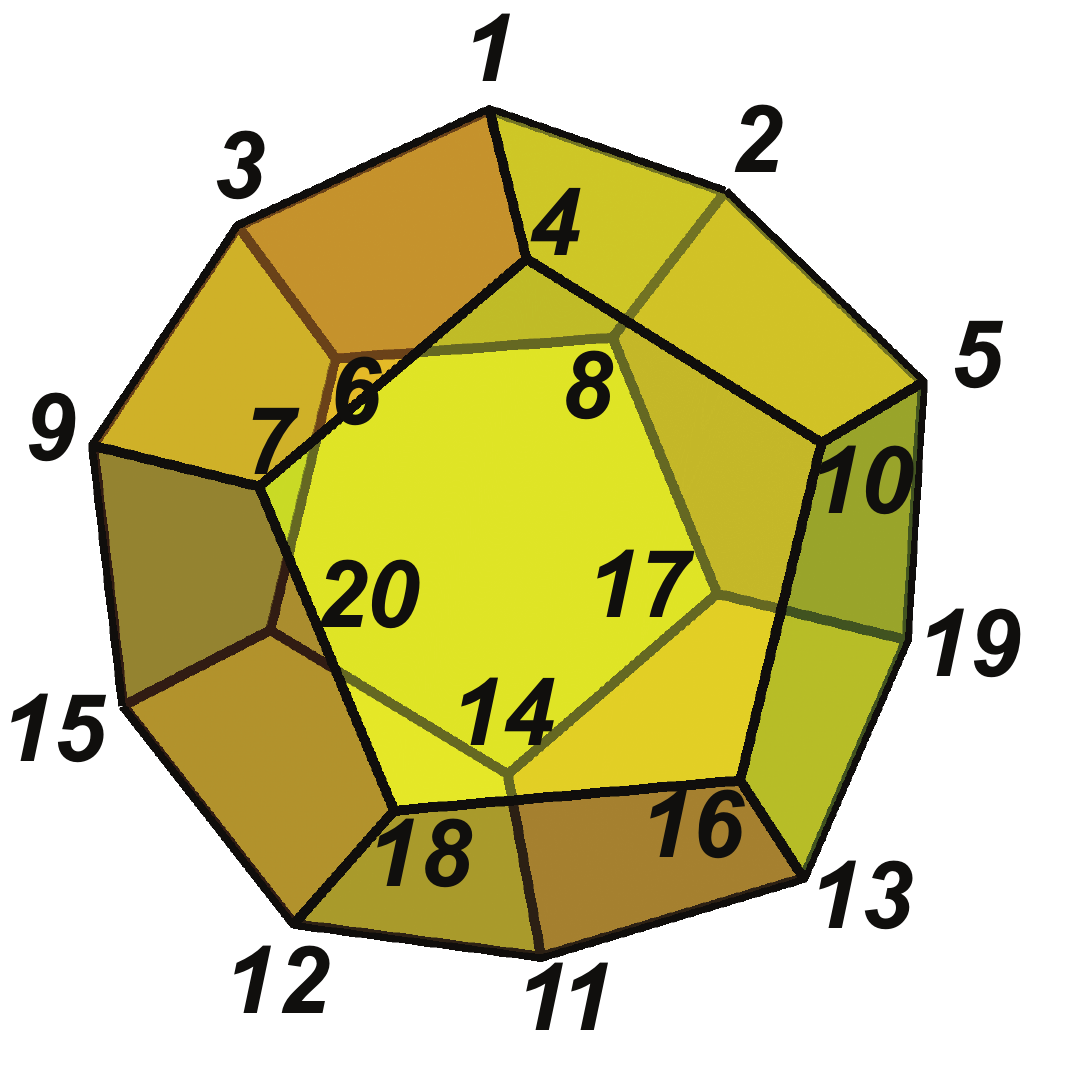}
        \label{Fig:Int:Dodecahedron}}

    \caption{Five Platonic Solids. }
    \label{Fig:Int:PlatonicSolid}
\end{figure}

\end{document}